\documentclass[11pt]{article}

\usepackage{amsmath,amssymb,amsfonts,textcomp,amsthm,xifthen,psfrag,graphicx,color,MnSymbol}
\usepackage{svg}
\usepackage[T1]{fontenc}
\usepackage{hyperref}
\usepackage{placeins}

\usepackage{pgfplots}
\usepgfplotslibrary{colorbrewer}
\pgfplotsset{compat=newest}
\usepgfplotslibrary{groupplots}
\usepgfplotslibrary{external}
\definecolor{navyblue}{rgb}{0.0, 0.0, 0.5}

\date{}

\newtheorem{theorem}{Theorem}
\newtheorem{lemma}[theorem]{Lemma}
\newtheorem{cor}[theorem]{Corollary}
\newtheorem{prop}[theorem]{Proposition}
\newtheorem{remark}[theorem]{Remark}

\theoremstyle{definition} 

\newcommand{\<}{\langle{}}
\renewcommand{\>}{\rangle}

\newcommand{\diam}{\mathrm{diam}}

\newcommand{\wat}{\widehat}

\newcommand{\G}{\Gamma}
\newcommand{\bt}{{\boldsymbol{t}}}
\newcommand{\bn}{{\boldsymbol{n}}}
\newcommand{\nn}{{\boldsymbol{n}}}

\newcommand{\ip}[2]{\llangle#1\hspace*{.5mm},#2\rrangle}
\newcommand{\dual}[2]{\<#1\hspace*{.5mm},#2\>}
\newcommand{\vdual}[2]{(#1\hspace*{.5mm},#2)}

\newcommand{\dualGC}[2]{\<#1\hspace*{.5mm},#2\>}
\newcommand{\dualGCa}[2]{\<#1\hspace*{.5mm},#2\>^*}

\newcommand{\grad}{\nabla}
\DeclareMathOperator{\Grad}{\boldsymbol{\nabla}}

\let\div\undefined
\DeclareMathOperator{\div}{{\rm div}}
\DeclareMathOperator{\Div}{{\rm\bf div}}
\DeclareMathOperator{\dDiv}{\div\Div}

\DeclareMathOperator{\curl}{curl}
\DeclareMathOperator{\Curl}{\mathbf{curl}}
\DeclareMathOperator{\scurl}{curl^s}
\DeclareMathOperator{\rot}{rot}

\DeclareMathOperator{\aGrad}{\boldsymbol{\nabla^{\textmd{*}}\!}} 
\DeclareMathOperator{\aCurl}{\mathbf{curl^*\!}} 

\DeclareMathOperator{\GCurl}{\Grad\Curl}
\DeclareMathOperator{\GCurla}{\aCurl\aGrad} 

\newcommand{\HGradcurl}[1]{{H(\Grad\curl,#1)}}
\newcommand{\HrotDiv}[1]{{\HH(\rot\Div,#1)}}

\newcommand{\HGradCurl}[1]{{\bH(\Grad\Curl,#1)}}
\newcommand{\HGradCurla}[1]{{\HH(\aCurl\aGrad,#1)}}
\newcommand{\HGradCurlz}[1]{{\bH_0(\Grad\Curl,#1)}}
\newcommand{\HGC}{\wat\bH}
\newcommand{\HGCz}{\wat\bH_0}
\newcommand{\HGCa}{\wat\bH^*}

\newcommand{\HH}{\mathbb{H}}
\newcommand{\bH}{\boldsymbol{H}}

\newcommand{\HDiv}[1]{{\HH(\Div,#1)}}

\newcommand{\HdDiv}[1]{{\HH(\dDiv,#1)}}

\def\GG{\boldsymbol{G}}
\def\MM{\boldsymbol{M}}
\def\SS{\boldsymbol{S}}

\def\II{\boldsymbol{I}}
\def\QQ{\boldsymbol{Q}}
\def\PP{\boldsymbol{P}}

\def\bQ{\boldsymbol{Q}}
\def\bR{\boldsymbol{R}}

\def\bq{\boldsymbol{q}}

\newcommand{\cD}{\mathcal{D}}
\newcommand{\cbD}{\boldsymbol{\mathcal{D}}}
\newcommand{\cDD}{\mathbb{D}}

\newcommand{\bL}{\ensuremath{\boldsymbol{L}}}
\newcommand{\LL}{\ensuremath{\mathbb{L}}}
\def\tq{\wat{\boldsymbol{q}}}

\def\tbv{\wat{\boldsymbol{v}}}

\def\tbw{\wat{\boldsymbol{w}}}

\def\tQ{\wat{\boldsymbol{q}}}
\def\tP{\wat{\boldsymbol{p}}}
\def\tu{\wat{\boldsymbol{u}}}

\newcommand{\bP}{\boldsymbol{P}}
\newcommand{\ff}{\boldsymbol{f}}
\newcommand{\bg}{\boldsymbol{g}}

\newcommand{\bp}{\boldsymbol{p}}
\newcommand{\dbv}{\boldsymbol{\delta\!v}}

\newcommand{\dQQ}{{\boldsymbol{\delta}\!\QQ}}
\newcommand{\bv}{\boldsymbol{v}}
\newcommand{\bu}{\boldsymbol{u}}

\newcommand{\uu}{\mathfrak{u}}

\newcommand{\dvv}{\delta\!\vv}
\newcommand{\duu}{\delta\!\uu}
\newcommand{\vv}{\mathfrak{v}}
\newcommand{\ww}{\mathfrak{w}}

\newcommand{\UU}{\ensuremath{\mathfrak{U}}}
\newcommand{\VV}{\ensuremath{\mathfrak{V}}}

\newcommand{\bphi}{\ensuremath{\boldsymbol{\phi}}}

\newcommand{\bx}{\ensuremath{\boldsymbol{x}}}

\newcommand{\deltabv}{{\boldsymbol{\delta}\!\bv}}

\newcommand{\deltav}{\delta\!v}

\newcommand{\deltaQQ}{\boldsymbol{\delta}\!\mathbf{Q}}

\newcommand{\traceDD}[1]{\mathrm{tr}_{#1}^{\mathrm{dDiv}}} 
\newcommand{\traceGG}[1]{\mathrm{tr}_{#1}^{\mathrm{Ggrad}}} 

\newcommand{\tr}[1]{\mathrm{tr}_{#1}^{\mathrm{GCurl}}}
\newcommand{\tra}[1]{\mathrm{tr}_{#1}^{\mathrm{GCurl}^*}}

\newcommand{\jump}[1]{[#1]}

\newcommand{\ttt}{{\mathfrak{T}}}

\newcommand{\di}{d}
\newcommand{\R}{\ensuremath{\mathbb{R}}}

\newcommand{\cB}{\ensuremath{\mathcal{B}}}
\newcommand{\cT}{\ensuremath{\mathcal{T}}}

\newcommand{\cS}{\ensuremath{\mathcal{S}}}

\newcommand{\OO}{\ensuremath{\mathcal{O}}}
\newcommand{\cE}{\ensuremath{\mathcal{E}}}
\newcommand{\cN}{\ensuremath{\mathcal{N}}}

\newcommand{\eeta}{{\boldsymbol\eta}}

\title{A DPG method for the quad-curl problem
\thanks{Supported by ANID-Chile through FONDECYT projects 1190009, 1210391}\\
{\small
Dedicated to Professor Leszek F. Demkowicz on the occasion of his 70$^\mathrm{th}$ birthday}
\author{
Thomas~F\"uhrer$^\dagger$
\and
Pablo Herrera$^\dagger$
\and
Norbert Heuer\thanks{
Facultad de Matem\'aticas, Pontificia Universidad Cat\'olica de Chile,
Avenida Vicu\~na Mackenna 4860, Santiago, Chile,
email: {\tt \{tofuhrer,pcherrera,nheuer\}@mat.uc.cl}}}}

\begin{document}
\maketitle
\begin{abstract}
We derive an ultraweak variational formulation of the quad-curl problem in two and three
dimensions. We present a discontinuous Petrov--Galerkin (DPG) method for its approximation
and prove its quasi-optimal convergence. We illustrate how this method can be applied to the
Stokes problem in two dimensions, after an application of the curl operator to eliminate the
pressure variable. In this way, DPG techniques known from Kirchhoff--Love plates can be used.
We present an a priori error estimate that improves a previous approximation result
for effective shear forces by using a less restrictive regularity assumption.
Numerical experiments illustrate our findings.

\end{abstract}

\section{Introduction}

In recent years, there has been an increased interest in the numerical analysis of quad-curl
problems, see
\cite{
BrennerSS_17_HDM,
CakoniCMS_10_IES,
CaoCH_EAD,
ChenQX_21_AIP,
HongHSX_12_DGM,
MonkS_12_FEM,
ZhangZ_CCC,
Zhang_18_MSQ,
ZhengHX_11_NFE}.
Such problems appear, for instance, in Maxwell transmission eigenvalue problems
\cite{Haddar_04_ITP} and magneto-hydrodynamics \cite{Biskamp_00_MRP}.
By relation $\Curl^4\bu=(\grad\times)^4\bu=-\Curl\Div\Grad\Curl\bu$
for vector fields $\bu$ in three dimensions, a canonical variational formulation
leads to considering spaces $\bH(\Curl^2,\Omega)$ and $\bH(\Grad\Curl,\Omega)$
of $L_2(\Omega)$ vector functions $\bu$ with $\Curl^2\bu\in L_2(\Omega)^3$ and
$\Grad\Curl\bu\in L_2(\Omega)^{3\times 3}$, respectively. Inherent regularities
are different and depend on the imposed boundary conditions. The construction and analysis of
conforming discretizations of these spaces are non-trivial and subject of ongoing research.

In this paper we continue our study of the discontinuous Petrov--Galerkin method
with optimal test functions (DPG method) for fourth-order problems.
Previously, we considered plate and shell problems of Kirchhoff--Love type in
\cite{FuehrerHH_21_TOB,FuehrerH_19_FDD,FuehrerHN_19_UFK,FuehrerHN_22_DMS}
and a problem of divergence type in \cite{FuehrerHH_22_DMF}.
The DPG method is rather a framework, proposed by Demkowicz and Gopalakrishnan,
that combines in its standard form ultraweak variational formulations,
product (or ``broken'') test spaces, the use of independent
trace variables and optimal test functions, see \cite{DemkowiczG_14_ODM} for an
early overview. Main advantages are the inherent discrete inf-sup stability,
induced stiffness matrices that are symmetric, positive definite (for real problems),
and the fact that there are built-in local error estimators with (optional) induced adaptivity.

The fact that ultraweak formulations have field variables that are only $L_2$-regular
becomes an additional advantage when aiming at conformity.
Continuity constraints on discretizations are easier to handle for traces
than for field variables in the corresponding energy spaces
(like $\bH(\Curl^2,\Omega)$ and $\bH(\Grad\Curl,\Omega)$ mentioned before).
This allowed us to provide conforming approximations for non-convex Kirchhoff--Love plates
that include bending moments and the effective shear force, cf.~\cite{FuehrerHN_19_UFK},
with degrees of freedom that require little more than the standard energy regularity.

In this paper we consider the quad-curl problem in the form of operator
$-\Curl\Div\Grad\Curl$, both in three and two space dimensions.
This model is introduced in the next section.
In Section~\ref{sec_DPG} we develop a variational formulation for the model problem
that is based on a second-order system for the variables $\bu$ and $\PP:=-\Grad\Curl\bu$.
We state its well-posedness and the quasi-optimal convergence of the corresponding DPG scheme
(Theorem~\ref{thm}). The analysis of the variational formulation requires to introduce and study
some trace operators. This is done in the preceding Section~\ref{sec_traces}.
A proof of Theorem~\ref{thm} is given in Section~\ref{sec_pf}.
A fully discrete analysis, including the approximation of optimal test functions
and a priori error estimates, is presented for two space dimensions in Section~\ref{sec_Stokes}.
For illustration, we consider the Stokes problem with velocity $\bu$ represented as
the curl of a scalar function $u$. The advantage of such a formulation is that
it provides a pressure-robust approximation. Though the velocity has to be
determined in a post-processing step. This can be done in a piecewise or continuous manner,
giving a piecewise or globally divergence-free velocity approximation.
Here, we present an estimate for a piecewise approximation
(Corollary~\ref{cor_DPG_Stokes} in Section~\ref{sec_Stokes_DPG}).
We note that in \cite{RobertsBTD_14_DMS}, Roberts \emph{et al.} present and analyze a DPG scheme
for the Stokes problem that is based on a first-order system, a standard DPG approach for second-order
problems. In this way, both the velocity and pressure are directly approximated.
Furthermore, in \cite{EllisDC_14_LCD}, Ellis \emph{et al.} propose a DPG technique
for fluid problems that is locally conservative.

Our transformation of the Stokes problem leads to a formulation that represents the
bi-Laplacian. In fact, $-\Curl\Div\Grad\Curl\bu=\Delta^2\bu$ in two dimensions with
scalar function $\bu$. We make use of the implied relations of trace operators to apply,
to the Stokes problem, the DPG discretization for the Kirchhoff--Love plate bending problem in
\cite{FuehrerHN_19_UFK,FuehrerH_19_FDD}, with two important differences.
First, in the transformed Stokes problem, the right-hand side function becomes
an $H^{-1}(\Omega)$-load, a functional acting on $H^1_0(\Omega)$.
We follow the technique from \cite{FuehrerHK_22_MSO} to deal with such a load.
Second, the a priori error estimate derived in \cite{FuehrerHN_19_UFK} requires the solution
to be $H^4(\Omega)$-regular. This is usually not satisfied even for
convex polygonal plates. Here, we refine the a priori estimate to functions which
are $H^3(\Omega)$-regular (Proposition~\ref{prop_DPG_Stokes_L2}), and thus improve
the analysis in \cite{FuehrerHN_19_UFK} for Kirchhoff--Love plates.
In  Section~\ref{sec_num} we present several numerical examples for the Stokes problem.

In the following, we use the generic notation $a\lesssim b$ to indicate that $a\le cb$ with a
constant $c>0$ that is independent of involved functions and other data, except noted otherwise.
Notation $a\gtrsim b$ is used analogously, and $a\simeq b$ means that $a\lesssim b$ and
$b\lesssim a$.

\section{Model problem}

We consider a bounded, simply connected Lipschitz domain $\Omega\subset\R^\di$ ($\di\in\{2,3\}$)
with boundary $\G=\partial\Omega$ and exterior unit normal vector $\bn$ along $\G$, and
use the standard differential operators
\begin{align*}
   \Grad\begin{pmatrix}v_1\\\vdots\\v_\di\end{pmatrix}
   :=\begin{pmatrix} \grad v_1, \ldots, \grad v_\di\end{pmatrix}^\top,\quad
   \Div\begin{pmatrix}\bq_1, \ldots, \bq_\di\end{pmatrix}^\top
   :=\begin{pmatrix} \div\bq_1\\ \vdots\\ \div\bq_\di\end{pmatrix},\quad
   \Curl\bv:= \begin{cases} \grad\times\bv & (\di=3),\\ (\grad v)^\perp & (\di=2),\end{cases}
\end{align*}
where $\grad$ and $\div$ are the standard gradient and divergence operators, and
$\begin{pmatrix}v_1\\v_2\end{pmatrix}^\perp := \begin{pmatrix}v_2\\-v_1\end{pmatrix}$.
For ease of presentation we introduce the formal adjoint operators $\aGrad$ and $\aCurl$
of $\Grad$ and $\Curl$, respectively. Of course, $\aGrad=-\Div$ and
\[
    \aCurl\bv = \begin{cases} \Curl\bv & (\di=3),\\ \rot\bv:=\div(\bv^\perp) & (\di=2).\end{cases}
\]
Now, for a given vector function $\ff$ and a constant $\gamma>0$, our model problem in three dimensions is
\begin{align} \label{prob3d}
   -\Curl\Div\Grad\Curl\bu + \gamma\bu = \ff\quad\text{in}\ \Omega,\quad
   \bn\times\bu=\Curl\bu = 0\quad\text{on}\ \Gamma.
\end{align}
In two dimensions, and $f$ being a scalar function, the problem reads
\begin{align} \label{prob2d}
   -\rot\Div\Grad\curl u + \gamma u = f\quad\text{in}\ \Omega,\quad
   u=0,\ \curl u = 0\quad\text{on}\ \Gamma.
\end{align}
In this case, $\gamma=0$ is permitted.
We note that $\mathrm{tr}(\Grad\Curl\bu)=\div\Curl\bu=0$  for $\di=2,3$
(tr referring to the trace of matrices), and
\begin{equation} \label{rep}
   -\Curl\Div\Grad\Curl\bu = \Delta^2\bu - \grad\Delta\div\bu\quad(\di=3),\quad
    -\rot\Div\Grad\curl u = \Delta^2 u\quad(\di=2).
\end{equation}
Let us introduce several Sobolev spaces. For a subdomain $\omega\subset\Omega$,
$L_2(\omega)$, $H^1(\omega)$, $H^2(\omega)$ are the standard spaces, and
$\bL_2(\omega):=L_2(\omega)^3$, $\LL_2(\omega):=L_2(\omega)^{\di\times\di}$,
$\bH^1(\omega)=H^1(\omega)^3$, $\HH^1(\omega):=H^1(\omega)^{\di\times\di}$. Furthermore, we introduce
\begin{align*}
   &\LL_2^s(\omega):=\{\QQ\in\LL_2(\omega);\; \QQ=\QQ^\top\},\quad
   \LL_2^\perp(\omega):=\{\QQ\in\LL_2(\omega);\; \mathrm{tr}(\QQ)=0\},\\
   &\HGradCurl{\omega}:=\{\bv\in\bL_2(\omega);\; \GCurl\bv\in\LL_2(\omega)\},\\
   &\HGradCurla{\omega}:=
   \begin{cases} \{\QQ\in\LL_2^\perp(\omega);\; \Curl\Div\QQ\in\bL_2(\omega)\} & (\di=3),\\
                 \{\QQ\in\LL_2^\perp(\omega);\; \rot\Div\QQ\in L_2(\omega)\} & (\di=2).
   \end{cases}
\end{align*}
In $\HGradCurl{\omega}$, $\HGradCurla{\omega}$
we use the respective graph norms $\|\cdot\|_{\Grad\Curl,\omega}$, $\|\cdot\|_{\aCurl\aGrad,\omega}$.
The $L_2(\omega)$, $\bL_2(\omega)$, and $\LL_2(\omega)$ inner products are generically denoted
by $\vdual{\cdot}{\cdot}_\omega$, and the norms by $\|\cdot\|_\omega$.
Generally, we drop the index $\omega$ when $\omega=\Omega$.

We also need the sets $\cD(\Omega)$, $\cbD(\Omega)$ of smooth scalar and vector functions,
respectively, with compact support in $\Omega$. Furthermore,
$\cDD^\perp(\Omega)$ refers to the space of smooth tensors with zero trace and compact support
in $\Omega$, and $\cDD^\perp(\overline\Omega)$ denotes the restriction to $\Omega$ of $C^\infty(\R^\di)$
$\di\times\di$-tensor functions with zero trace.

In the following we consider problems \eqref{prob3d}, \eqref{prob2d} in one generic form
to analyze them together within a single framework.
To this end, we slightly abuse our notation, and tacitly identify
\begin{align*}
   \Curl = \begin{cases} \Curl,\\ \curl,\end{cases}\
   \bL_2(\omega) = \begin{cases} \bL_2(\omega),\\ L_2(\omega),\end{cases}\
   \bH^1(\omega) = \begin{cases} \bH^1(\omega),\\ H^1(\omega),\end{cases}\
   \cbD(\Omega) = \begin{cases} \cbD(\Omega) & \quad\text{if}\ \di=3,\\
                                \cD(\Omega) & \quad\text{if}\ \di=2.\end{cases}
\end{align*}
Furthermore, if $\di=3$, bold-face symbols $\bu,\bv,\ff,\ldots$ indicate vector-valued functions
which are understood to be scalar for $\di=2$. In the latter case,
$\HGradCurl{\omega}=H^2(\omega)$ and $\HGradCurla{\omega}=\HH(\rot\Div,\omega)$ (not explicitly introduced).

To represent the homogeneous Dirichlet conditions of \eqref{prob3d} and \eqref{prob2d},
we introduce $\HGradCurlz{\Omega}$ as the completion of $\cbD(\Omega)$ with respect to norm
$\|\cdot\|_{\GCurl}$. We will see that
\begin{align*}
   &\HGradCurlz{\Omega} =
   \begin{cases}
      \{\bv\in \HGradCurl{\omega};\; \bn\times\bu|_\G=\Curl\bu|_\G = 0\} & (\di=3),\\
      \{v\in H^2(\Omega);\; u|_\G=0,\ \grad\!u|_\G=0\} = H^2_0(\Omega)   & (\di=2),
   \end{cases}
\end{align*}
cf.~Proposition~\ref{prop_zero} below.

Finally, the generic form of problems \eqref{prob3d}, \eqref{prob2d} reads
\begin{align} \label{prob}
   \bu\in\HGradCurlz{\Omega}:\quad
   \GCurla\GCurl\bu + \gamma\bu = \ff\quad\text{in}\ \Omega.
\end{align}
Recall that $\gamma$ is a fixed non-negative constant, and $\gamma>0$ when $\di=3$.

\section{Trace operators} \label{sec_traces}

Let $\cT=\{T\}$ be a mesh of non-intersecting Lipschitz polyhedra (polygons in $\R^2$),
$\overline\Omega=\cup\{\overline T;\; T\in\cT\}$. The union of boundaries generates the skeleton
$\cS:=\{\partial T;\;T\in\cT\}$. We need the corresponding product spaces
\(
   \HGradCurl{\cT},\quad \HGradCurla{\cT}
\)
with canonical product norms $\|\cdot\|_{\GCurl,\cT}$ and $\|\cdot\|_{\GCurla,\cT}$, respectively.
Throughout this paper we identify elements of $\cT$-product spaces and piecewise defined functions
on $\cT$. The generic $L_2(\cT)$-bilinear form is denoted by $\vdual{\cdot}{\cdot}_\cT$.

We define the following trace operators with support on $\cS$,
\begin{align*}
   \tr{}:\;
   &\begin{cases}
      \HGradCurl{\Omega} &\to\ \HGradCurla{\cT}'\\
      \qquad \bv    &\mapsto\ \dualGC{\tr{}(\bv)}{\QQ}_\cS
                    := \vdual{\GCurl\bv}{\QQ} - \vdual{\bv}{\GCurla\QQ}_\cT
   \end{cases},
   \\
   \tra{}:\;
   &\begin{cases}
      \HGradCurla{\Omega} &\to \HGradCurl{\cT}' \\
      \qquad \QQ    &\mapsto\ \dualGCa{\tra{}(\QQ)}{\bv}_\cS
                    := \vdual{\GCurl\bv}{\QQ}_\cT - \vdual{\bv}{\GCurla\QQ}
   \end{cases}.
\end{align*}
The particular cases mapping to $\HGradCurla{\Omega}'$ and $\HGradCurl{\Omega}'$, respectively,
are denoted as $\tr{\Gamma}$ and $\tra{\Gamma}$:
\begin{align*}
   \dualGC{\tr{\Gamma}(\bv)}{\QQ}_\G := \dualGCa{\tra{\Gamma}(\QQ)}{\bv}_\G
   := \vdual{\GCurl\bv}{\QQ} - \vdual{\bv}{\GCurla\QQ}
\end{align*}
for $\bv\in \HGradCurl{\Omega}$, $\QQ\in\HGradCurla{\Omega}$.

\begin{remark} \label{rem_trace}
We note that, for sufficiently smooth tensor functions $\QQ$,
\begin{align*}
   \dualGC{\tr{\Gamma}(\bv)}{\QQ}_\G = \dualGCa{\tra{\Gamma}(\QQ)}{\bv}_\G
   = \begin{cases}
      \dual{\QQ\bn}{\Curl\bv}_{L_2(\G)} - \dual{\Div\QQ}{\bn\times\bv}_{L_2(\G)} & (\di=3),\\
      \dual{\QQ\bn}{\Curl\bv}_{L_2(\G)} + \dual{\bt\cdot\Div\QQ}{\bv}_{L_2(\G)} & (\di=2).
   \end{cases}
\end{align*}
Here, $\dual{\cdot}{\cdot}_{L_2(\G)}$ denotes the generic duality between negative and positive-order
Sobolev spaces on $\G$ (in any order), with $L_2(\G)$ as pivot space, and $\bt$ is the unit tangential
vector along $\G$ in mathematically positive orientation.
In three dimensions, functions $\bv\in\HGradCurl{\Omega}$
satisfy $\bv\in\bL_2(\Omega)$ and $\Curl\bv\in\bH^1(\Omega)$.
Therefore, their traces $\Curl\bv|_\G\in \bH^{1/2}(\G)$ and $\bn\times\bv|_\G\in\bH^{-1/2}(\G)$
are well defined in the canonical way, and give rise to bounded operators.
In two dimensions, it is clear that $\bv\in H^2(\Omega)$ has bounded traces $(\bv,\grad\bv)|_\G$,
cf.~\cite{Grisvard_85_EPN,CostabelD_96_IBS,FuehrerHN_19_UFK}.
On the other hand, traces $\QQ\bn|_\G$ and $\Div\QQ|_\G$ are not well defined for
$\QQ\in\HGradCurla{\Omega}$ since $\Div\QQ=-\aGrad\QQ\in\bL_2(\Omega)$ is not guaranteed.
We refer to \cite[Remark~3.1 and \S 6.2.2.]{FuehrerHN_19_UFK} for further details.
\end{remark}

Trace operators $\tr{}$, $\tra{}$ give rise to the following trace spaces,
\begin{align*}
   &\HGC(\cS):=\tr{}\bigl(\HGradCurl{\Omega}\bigr),\quad
   \HGCz(\cS):=\tr{}\bigl(\HGradCurlz{\Omega}\bigr),\\
   &\HGCa(\cS):=\tra{}\bigl(\HGradCurla{\Omega}\bigr),
\end{align*}
furnished with their respective operator and trace norms,
\begin{align*}
   \|\tbv\|_{(\GCurla,\cT)'}
   &:= \sup \{\dualGC{\tbv}{\QQ}_\cS;\; \QQ\in\HGradCurla{\cT},\ \|\QQ\|_{\GCurla,\cT}=1\},\\
   \|\tbv\|_{\inf}
   &:= \inf\{\|\bv\|_{\GCurl};\; \bv\in\HGradCurl{\Omega},\ \tr{}(\bv)=\tbv\}
   \qquad \bigl(\tbv\in \HGC(\cS)\bigr),\\
   \|\tQ\|_{(\GCurl,\cT)'}
   &:= \sup \{\dualGCa{\tQ}{\bv}_\cS;\; \bv\in\HGradCurl{\cT},\ \|\bv\|_{\GCurl,\cT}=1\},\\
   \|\tQ\|_{\inf^*}
   &:= \inf\{\|\QQ\|_{\GCurla};\; \QQ\in\HGradCurla{\Omega},\ \tra{}(\QQ)=\tQ\}
   \qquad \bigl(\tQ\in \HGCa(\cS)\bigr).
\end{align*}
Here, the dualities $\dualGC{\cdot}{\cdot}_\cS$ and $\dualGCa{\cdot}{\cdot}_\cS$ are defined
to be consistent with the definition of the corresponding trace operator. That is,
\begin{align*}
   &\dualGC{\tbv}{\QQ}_\cS := \dualGC{\bv}{\QQ}_\cS
   \quad\text{for any}\ \bv\in(\tr{})^{-1}(\tbv)\quad
   \bigl(\tbv\in\HGC(\cS),\ \QQ\in\HGradCurla{\cT}\bigr),\\
   &\dualGCa{\tQ}{\bv}_\cS := \dualGCa{\QQ}{\bv}_\cS
   \quad\text{for any}\ \QQ\in(\tra{})^{-1}(\tQ)\quad
   \bigl(\tQ\in\HGCa(\cS),\ \bv\in\HGradCurl{\cT}\bigr).
\end{align*}

\begin{prop} \label{prop_traces}
The following norm identities hold true,
\begin{align*}
   \|\tbv\|_{\inf} &= \|\tbv\|_{(\GCurla,\cT)'}\quad (\tbv\in \HGC(\cS)),\\
   \|\tQ\|_{\inf^*} &= \|\tQ\|_{(\GCurl,\cT)'}\quad (\tQ\in \HGCa(\cS)).
\end{align*}
\end{prop}

\begin{proof}
The two relations are proved in the canonical way. For first-order operators (traces of $H^1$ and
$H(\div)$) see~\cite[Lemma~2.2, Theorem~2.3]{CarstensenDG_16_BSF}. For the general case
there is a framework proposed in \cite[Appendix A]{DemkowiczGNS_17_SDM}, though building upon
certain density properties. For our case of second-order traces (stemming from second-order
derivatives), there is a canonical procedure given in
\cite[Proofs of Lemma~3.2, Proposition 3.5]{FuehrerHN_19_UFK}. Let us recall the essential steps.
Bounds $\|\cdot\|_{(\GCurla,\cT)'}\le \|\cdot\|_{\inf}$ and
$\|\cdot\|_{(\GCurl,\cT)'}\le \|\cdot\|_{\inf^*}$ are due to the Cauchy--Schwarz inequality
applied to the $L_2$-dualities of the variational definitions of the traces.

To prove the bound $\|\cdot\|_{\inf}\le \|\cdot\|_{(\GCurla,\cT)'}$ define, for a given
$\tbv\in \HGC(\cS)\setminus\{0\}$, functions $\QQ\in\HGradCurla{\cT}$ and $\bv\in\HGradCurl{\Omega}$
as the solutions to
\[
   \vdual{\QQ}{\dQQ} + \vdual{\GCurla\QQ}{\GCurla\dQQ}_\cT
   = \dualGC{\tbv}{\dQQ}_\cS\quad\forall\dQQ\in\HGradCurla{\cT}
\]
and
\[
   \vdual{\bv}{\dbv} + \vdual{\GCurl\bv}{\GCurl\dbv}
   = \dualGC{\tr{}(\dbv)}{\QQ}_\cS\quad\forall\dbv\in\HGradCurl{\Omega},
\]
respectively. One then proves that $\bv=-\GCurla\QQ$ on every $T\in\cT$,
$\tr{}(\bv)=\tbv$, and $\|\QQ\|_{\GCurla,\cT}^2=\|\bv\|_{\GCurl}^2=\dualGC{\tbv}{\QQ}_\cS$,
leading to the desired estimate,
\[
   \|\tbv\|_{(\GCurla,\cT)'}
   \ge \frac {\dualGC{\tbv}{\QQ}_\cS}{\|\QQ\|_{\GCurla,\cT}}
   = \|\bv\|_{\GCurl} = \|\tbv\|_{\inf}.
\]
The remaining inequality is proved completely analogously, by the construction chain
\[
   \HGCa(\cS)\setminus\{0\}\ni\tQ\mapsto \bv\in\HGradCurl{\cT}\mapsto \QQ\in\HGradCurla{\Omega}
\]
with $\tra{}(\QQ)=\tQ$ and
$\|\QQ\|_{\GCurla}^2=\|\bv\|_{\GCurl,\cT}^2=\dualGCa{\tQ}{\bv}_\cS$, giving
\[
   \|\tQ\|_{(\GCurl,\cT)'}
   \ge \frac {\dualGCa{\tQ}{\bv}_\cS}{\|\bv\|_{\GCurl,\cT}}
   = \|\QQ\|_{\GCurla} = \|\tQ\|_{\inf^*}.
\]
\end{proof}

\begin{prop} \label{prop_dense}
A function $\bv\in\HGradCurl{\Omega}$ satisfies
\[
   \bv\in\HGradCurlz{\Omega}\quad\Leftrightarrow\quad
   \dualGC{\tr{\Gamma}(\bv)}{\QQ}_\G=0\ \forall\QQ\in\HGradCurla{\Omega}.
\]
Furthermore, $\cDD^\perp(\overline\Omega)$ is dense in $\HGradCurla{\Omega}$.
\end{prop}

\begin{proof}
We split the proof into two steps. In step 1 we show that the statement holds true when replacing
space $\HGradCurla{\Omega}$ with $\cDD^\perp(\overline\Omega)$. In the second step we show that
$\cDD^\perp(\overline\Omega)$ is dense in $\HGradCurla{\Omega}$.

{\bf Step 1.}
The fact that $\dualGC{\tr{\Gamma}(\bv)}{\QQ}_\G=0$
$\forall\QQ\in\cDD^\perp(\overline\Omega)$ holds for any $\bv\in\HGradCurlz{\Omega}$
is immediate by the density of $\cbD(\Omega)$ in $\HGradCurlz{\Omega}$ and integration by parts.
The reverse implication makes use of the fact that any $\bv\in\HGradCurl{\Omega}$
with $\dualGC{\tr{\Gamma}(\bv)}{\QQ}_\G=0$ $\forall\QQ\in\cDD^\perp(\overline\Omega)$
can be extended by zero to an element $\tilde\bv\in\HGradCurl{\R^\di}$
so that an approximating sequence $(\bphi_k)_k\subset\cbD(\Omega)$
can be constructed by partition of unity and mollifier techniques. We refer to
\cite[Lemma 2.4]{GiraultR_86_FEM} for details.

{\bf Step 2:} Density of $\cDD^\perp(\overline\Omega)\subset \HGradCurla{\Omega}$.
We follow the standard procedure by showing that any $\QQ\in\HGradCurla{\Omega}$ orthogonal to
$\cDD^\perp(\overline\Omega)$ in $\HGradCurla{\Omega}$ vanishes,
cf., e.g., \cite[Theorem~3.26]{Monk_03_FEM}. In fact, given $\QQ\in\HGradCurla{\Omega}$ with
\[
   \vdual{\QQ}{\dQQ}+\vdual{\GCurla\QQ}{\GCurla\dQQ} = 0\quad\forall \dQQ\in\cDD^\perp(\overline\Omega)
\]
and defining $\bv:=\GCurla\QQ$, we find that $\GCurl\bv=-\QQ$
so that $\bv\in\HGradCurl{\Omega}$ and
\[
   \dualGC{\tr{\Gamma}(\bv)}{\dQQ}_\G
   = \vdual{\GCurl\bv}{\dQQ} - \vdual{\bv}{\GCurla\dQQ}
   = 0\quad\forall\dQQ\in\cDD^\perp(\overline\Omega).
\]
We conclude with step 1 that $\bv\in\HGradCurlz{\Omega}$. By definition
of the latter space there is a sequence $(\bphi_j)_j\subset \cbD(\Omega)$ that converges
in $\HGradCurl{\Omega}$ to $\bv$. Therefore,
\begin{align*}
   \vdual{\GCurla\QQ}{\GCurla\QQ} + \vdual{\QQ}{\QQ}
   &= \vdual{\bv}{\GCurla\QQ} - \vdual{\GCurl\bv}{\QQ}\\
   &= \lim_{j\to\infty} \vdual{\bphi_j}{\GCurla\QQ} - \vdual{\GCurl\bphi_j}{\QQ} = 0,
\end{align*}
that is, $\QQ=0$ as wanted.
\end{proof}

\begin{prop} \label{prop_zero}
The characterization
\[
   \HGradCurlz{\Omega}=
   \begin{cases}
      \{\bv\in \HGradCurl{\Omega};\; \Curl\bv|_\G=0,\ \bn\times\bv|_\G=0\} & (\di=3),\\
      \{v\in H^2(\Omega);\; v|_\G=0,\ \grad\! v|_\G=0\} & (\di=2)
   \end{cases}
\]
holds true.
\end{prop}

\begin{proof}
Having Proposition~\ref{prop_dense} at hand, the proof of this statement is canonical,
cf.~\cite[Theorem~3.33]{Monk_03_FEM}.
Indeed, $\HGradCurlz{\Omega}$ is a subspace of the space on the right-hand side of
the statement due to the closedness of the latter space (cf.~Remark~\ref{rem_trace})
and the fact that the canonical traces of any $\bphi\in\cbD(\Omega)$ vanish.
On the other hand, if $\bv\in\HGradCurl{\Omega}$ has zero canonical traces
then integration by parts shows that $\dualGC{\tr{\Gamma}(\bv)}{\QQ}_\G=0$
$\forall\QQ\in\cDD^\perp(\overline\Omega)$, that is,
$\bv\in \HGradCurlz{\Omega}$ by Proposition~\ref{prop_dense}.
This proof applies to two and three space dimensions.
Of course, in two dimensions, where $\HGradCurlz{\Omega}=H^2_0(\Omega)$,
the result is well known, see, e.g., \cite{Grisvard_85_EPN}.
\end{proof}

\begin{prop} \label{prop_jumps}
For given $\bv\in\HGradCurl{\cT}$ and $\QQ\in\HGradCurla{\cT}$ the following relations hold true,
\[
   \bv\in \HGradCurlz{\Omega} \quad\Leftrightarrow\quad
   \dualGCa{\tra{}(\deltaQQ)}{\bv}_\cS = 0
   \quad\forall \deltaQQ\in \HGradCurla{\Omega},
\]
\[
   \QQ\in \HGradCurla{\Omega} \quad\Leftrightarrow\quad
   \dualGC{\tr{}(\deltabv)}{\QQ}_\cS = 0
   \quad\forall \deltabv\in \HGradCurlz{\Omega}.
\]
\end{prop}

\begin{proof}
The proof uses standard techniques, cf.~\cite[Theorem~2.3]{CarstensenDG_16_BSF} and
\cite[Propositions~3.4,~3.8]{FuehrerHN_19_UFK}.

Given $\bv\in\cbD(\Omega)$, the distributional definition of derivatives shows that
$\dualGCa{\tra{}(\QQ)}{\bv}_\cS = \dualGC{\tr{}(\bv)}{\QQ}_\cS = \dualGC{\tr{\Gamma}(\bv)}{\QQ}_\G = 0$
for any $\QQ\in\HGradCurla{\Omega}$.
The density of $\cbD(\Omega)\subset \HGradCurlz{\Omega}$ then proves direction
``$\Rightarrow$'' of both statements.

Now, for given $\bv\in\HGradCurl{\cT}$ and $\QQ\in\HGradCurla{\cT}$
with $\dualGCa{\tra{}(\dQQ)}{\bv}_\cS=0$ $\forall\dQQ\in\cDD^\perp(\Omega)$ and
$\dualGC{\tr{}(\dbv)}{\QQ}_\cS = 0$ $\forall\dbv\in\cbD(\Omega)$,
the regularities $\bv\in\HGradCurl{\Omega}$, $\QQ\in \HGradCurla{\Omega}$
hold by the distributional definition of derivatives and the definition of trace operators
$\tr{}$, $\tra{}$.
Furthermore, $\bv\in\HGradCurlz{\Omega}$ by Proposition~\ref{prop_dense}.
%
%
\end{proof}

\section{DPG method} \label{sec_DPG}

The Petrov--Galerkin scheme for \eqref{prob} is based upon an ultraweak variational
formulation. We write \eqref{prob} as a second-order system by introducing
$\PP:=-\GCurl\bu$. Testing this relation with $\QQ\in\HGradCurla{\cT}$,
and the remaining equation with $\bv\in\HGradCurl{\cT}$, invoking trace operators $\tr{}$, $\tra{}$,
and choosing trace variables $\tu:=\tr{}(\bu)$, $\tP:=\tra{}(\PP)$, we obtain
\begin{align} \label{b}
   b(\bu,\PP,\tu,\tP;\bv,\QQ) &:=
   \vdual{\bu}{\gamma\bv+\GCurla\QQ}_\cT + \vdual{\PP}{\QQ-\GCurl\bv}_\cT
   + \dualGC{\tu}{\QQ}_\cS + \dualGCa{\tP}{\bv}_\cS\\ \nonumber
   &= L(\bv,\QQ) := \vdual{\ff}{\bv}.
\end{align}
Selecting spaces and (squared) norms
\begin{align*}
   &\UU:=\bL_2(\Omega)\times\LL_2^\perp(\Omega)\times\HGCz(\cS)\times \HGCa(\cS),
   &&\|(\bu,\PP,\tu,\tP)\|_\UU^2 := \|\bu\|^2 + \|\PP\|^2 + \|\tu\|_{\inf}^2 + \|\tP\|_{\inf^*}^2,\\
   &\VV:=\HGradCurl{\cT}\times\HGradCurla{\cT},
   &&\|(\bv,\QQ)\|_\VV^2 := \|\bv\|_{\GCurl,\cT}^2 + \|\QQ\|_{\GCurla,\cT}^2,
\end{align*}
the variational formulation of \eqref{prob} reads
\begin{equation} \label{VF}
   \uu:=(\bu,\PP,\tu,\tP)\in\UU:\quad
   b(\uu,\vv) = L(\vv)\quad\forall \vv=(\bv,\QQ)\in\VV.
\end{equation}
For a given finite-dimensional subspace $\UU_h\subset\UU$ the DPG scheme then is
\begin{equation} \label{DPG}
   \uu_h:=(\bu_h,\PP_h,\tu_h,\tP_h)\in\UU_h:\quad
   b(\uu_h,\vv) = L(\vv)\quad\forall \vv\in\ttt(\UU_h)
\end{equation}
with trial-to-test operator $\ttt:\;\UU\to\VV$ defined as
\[
   \ip{\ttt\ww}{\dvv}_\VV = b(\ww,\dvv)\quad\forall \dvv\in\VV,\ \ww\in\UU.
\]
Here, $\ip{\cdot}{\cdot}_\VV$ is the inner product of $\VV$ that induces norm $\|\cdot\|_\VV$.

\begin{theorem} \label{thm}
Given $\ff\in\bL_2(\Omega)$ and $\gamma\ge 0$ ($\gamma>0$ if $\di=3$), \eqref{VF} and \eqref{DPG}
are well posed, with respective solutions $\uu=(\bu,\PP,\tu,\tP)\in\UU$, $\uu_h\in\UU_h$ that satisfy
$\|\uu\|_\UU\lesssim \|\ff\|$ and
\[
   \|\uu-\uu_h\|_\UU \lesssim \inf\{\|\uu-\ww\|_\UU;\; \ww\in\UU_h\}.
\]
The hidden constants are independent of $\ff$, $\cT$ and $\UU_h$.
Furthermore, $\bu\in\HGradCurlz{\Omega}$, $\PP=-\GCurl\bu\in\HGradCurla{\Omega}$,
$\tr{}(\bu)=\tu$, $\tra{}(\PP)=\tP$, and $\bu$ solves problem \eqref{prob}.
\end{theorem}

\section{Proof of the main result} \label{sec_pf}

To prove Theorem~\ref{thm} we start with the well-posedness of variational formulation \eqref{VF}.
In a second step we give some details on the quasi-optimal convergence of DPG scheme \eqref{DPG}.

{\bf Well-posedness of \eqref{VF}.}
The well-posedness of formulation \eqref{VF} can be seen by the Babu\v{s}ka--Brezzi theory.
The boundedness of functional $L$ and bilinear form $b(\cdot,\cdot)$ is immediate.
Furthermore, noting that
\begin{align*}
   \VV_0 &:=
   \{(\bv,\QQ)\in \VV;\;
     \dualGCa{\tQ}{\bv}_\cS=0\ \forall\tQ\in\HGCa(\cS),\
     \dualGC{\tbv}{\QQ}_\cS=0\ \forall\tbv\in\HGCz(\cS)\}\\
   &= \HGradCurlz{\Omega}\times\HGradCurla{\Omega}
\end{align*}
by Proposition~\ref{prop_jumps}, inf-sup condition
\begin{align} \label{infsup}
   \sup_{0\not=\vv\in \VV} \frac{b(\uu;\vv)}{\|\vv\|_\VV}
   \gtrsim \|\uu\|_\UU \quad\forall\uu\in\UU
\end{align}
follows from the two conditions
\begin{align}
   \label{infsupb}
   \sup_{0\not=\vv\in \VV}
   \frac{b(\uu;\vv)}{\|\vv\|_\VV}
   \gtrsim
   \|\uu\|_{\UU} \quad\forall\uu=(0,0,\tu,\tP)\in\UU
\end{align}
and
\begin{align}
   \label{infsupa}
   \sup_{0\not=\bv\in \VV_0}
   &\frac{b(\uu;\vv)}{\|\vv\|_\VV}
   \gtrsim
   \|\uu\|_\UU \quad\forall\uu=(\bu,\PP,0,0)\in\UU,
\end{align}
cf.~\cite[Theorem~3.3]{CarstensenDG_16_BSF}.
By Proposition~\ref{prop_traces}, \eqref{infsupb} holds with constant $1$,
and \eqref{infsupa} holds due to the stability of the adjoint problem, that reads as follows:
\emph{given $\bg_1\in\bL_2(\Omega)$ and $\GG_2\in\LL_2^\perp(\Omega)$, find
$(\bv,\QQ)\in\VV_0$ such that}
\begin{align} \label{adj}
   \GCurla\QQ + \gamma\bv = \bg_1,\quad
   \QQ-\GCurl\bv = \GG_2.
\end{align}
Indeed, \eqref{adj} is well posed. Eliminating $\QQ$ we are left with
\[
   \GCurla\GCurl\bv + \gamma\bv = \bg_1-\GCurla\GG_2,
\]
in variational form: \emph{find $\bv\in\HGradCurlz{\Omega}$ such that}
\[
   \vdual{\GCurl\bv}{\GCurl\dbv} + \gamma\vdual{\bv}{\dbv}
   =
   \vdual{\bg_1}{\dbv} - \vdual{\GG_2}{\GCurl\dbv}
   \quad\forall\dbv\in\HGradCurlz{\Omega}.
\]
Clearly, using a Poincar\'e inequality for $\di=2$ when $\gamma=0$,
this is a well-posed problem with stable solution
$\|\bv\|_{\GCurl}\lesssim \|\bg_1\| + \|\GG_2\|$. Finally, setting $\QQ:=\GCurl\bv+\GG_2$,
$(\bv,\QQ)$ is the unique and stable solution of adjoint problem \eqref{adj},
\[
   \|\bv\|_{\GCurl} + \|\QQ\|_{\GCurla} \lesssim \|\bg_1\| + \|\GG_2\|.
\]
This bound gives
\[
   \bigl(\|\bu\|^2 + \|\PP\|^2\bigr)^{1/2}
   = \sup_{0\not=(\bg_1,\GG_2)\in\bL_2(\Omega)\times\LL_2(\Omega)}
      \frac {\vdual{\bu}{\bg_1} + \vdual{\PP}{\GG_2}}{\bigl(\|\bg_1\|^2 + \|\GG_2\|^2\bigr)^{1/2}}
   \lesssim
   \sup_{0\not=\vv\in\VV_0} \frac {b(\bu,\PP,0,0;\vv)}{\|\vv\|_\VV},
\]
which is \eqref{infsupa}.

It remains to confirm the final condition of the Babu\v{s}ka--Brezzi theory, injectivity
\[
   \vv\in\VV:\quad b(\duu,\vv)=0\ \forall\duu\in\UU\quad\Rightarrow\quad\vv=0.
\]
Indeed, given such a $\vv\in\VV$, Proposition~\ref{prop_jumps} implies that
$\vv\in\VV_0$. Therefore, $\vv=(\bv,\QQ)$ solves \eqref{adj} with $\bg_1=0$ and $\GG_2=0$.
The well-posedness of \eqref{adj} means that $\vv=0$.

We conclude that formulation \eqref{VF} is well posed and that its solution
$\uu=(\bu,\PP,\tu,\tP)$ satisfies the claimed stability. It is easy to verify
that $\bu\in\HGradCurlz{\Omega}$,
$\PP=-\GCurl\bu\in\HGradCurla{\Omega}$, $\tu=\tr{}(\bu)$, $\tP=\tra{}(\PP)$,
and that $\bu$ solves \eqref{prob}.

{\bf Quasi-uniform convergence of \eqref{DPG}.}
Let us denote by $\cB:\;\UU\to\VV'$ the operator $\uu\mapsto b(\uu,\cdot)$.
Since the DPG scheme delivers the minimizer $\uu_h\in\UU_h$ of the residual
$\|\cB(\cdot)-L\|_{\VV'}$, \eqref{DPG} is well posed
and converges optimally in the residual norm $\|\cB\cdot\|_{\VV'}$. The uniform boundedness
of operator $\cB$ and inf-sup property \eqref{infsup} mean that
$\|\cdot\|_\UU$ and $\|\cB\cdot\|_{\VV'}$ are uniformly equivalent norms. This gives
the quasi-optimal error estimate, and finishes the proof of Theorem~\ref{thm}.

\section{Application to the Stokes problem in $\R^2$} \label{sec_Stokes}

For a simply-connected, bounded Lipschitz domain $\Omega\subset\R^2$ the Stokes problem
with homogeneous Dirichlet condition reads
\begin{align*}
   -\Div\Grad\bu + \grad p = \ff,\ \div\bu=0\ \text{in}\ \Omega,\quad \bu=0\ \text{on}\ \G.
\end{align*}
Since $\bu$ is solenoidal and vanishes on $\Gamma$, it can be represented as $\bu=\curl u$ with $u\in H^1_0(\Omega)$.
The Stokes problem becomes
\begin{align} \label{Stokes}
   -\rot\Div\Grad\curl u = \rot\ff\ \text{in}\ \Omega,\quad u=0,\ \curl u=0\ \text{on}\ \G,
\end{align}
without pressure variable $p$.
Here and in the following, we use the specific notation for $\di=2$, for functions, operators, and spaces.
Problem~\eqref{Stokes} is \eqref{prob} with $\gamma=0$ and a right-hand side
function(al) in $H^{-1}(\Omega):=(H^1_0(\Omega))'$.
Note that the left-hand side operator in \eqref{Stokes} is another form of writing the bi-Laplacian,
cf.~\eqref{rep}. It is therefore
no surprise that our discrete setting will be based upon the discretization of a plate bending
problem. Specifically, we employ the discrete spaces
and techniques from \cite{FuehrerHN_19_UFK,FuehrerH_19_FDD} where a right-hand
side function from $L_2(\Omega)$ was considered. Alternatively, the methods
from \cite{FuehrerHH_21_TOB} apply.
Using the regularization approach from \cite{FuehrerHK_22_MSO}, we extend the setting to include our
$H^{-1}(\Omega)$-functional. Furthermore, we improve the a priori error analysis
from \cite{FuehrerHN_19_UFK} to a setting with solution $u\in H^3(\Omega)$
whereas, in the cited article, $u\in H^4(\Omega)$ was assumed.

To use the setting and technical details from \cite{FuehrerHN_19_UFK} and \cite{FuehrerH_19_FDD},
let us relate corresponding spaces and operators. We extend vector operator $(\cdot)^\perp$ to
tensors by defining column-wise
\[
   \QQ^\perp=\begin{pmatrix}a & b\\ c& d\end{pmatrix}^\perp
   := \begin{pmatrix}c & d\\ -a& -b\end{pmatrix}
   \quad\forall\QQ\in\LL_2(\omega),\ \omega\subset\Omega.
\]
We will need the following space and (squared) norm for a Lipschitz domain $\omega\subset\Omega$,
\begin{align*}
   \HdDiv{\omega}:=\{\QQ\in\LL_2^s(\omega);\; \div\Div\QQ\in L_2(\omega)\},\quad
   \|\QQ\|_{\div\Div,\omega}^2:=\|\QQ\|_\omega^2+\|\div\Div\QQ\|_\omega^2.
\end{align*}
We also use the corresponding product space and norm with index $\cT$ instead of $\omega$.
Finally, let us recall the following trace operators from \cite{FuehrerHN_19_UFK},
\begin{align*}
   \traceGG{}:\;H^2(\Omega)\to\HdDiv{\cT},\quad
   \dual{\traceGG{}(v)}{\QQ}_\cS:=\vdual{v}{\div\Div\QQ}_\cT-\vdual{\Grad\grad v}{\QQ},\\
   \traceDD{}:\;\HdDiv{\Omega}\to H^2(\cT),\quad
   \dual{\traceDD{}(\QQ)}{v}_\cS:=\vdual{v}{\div\Div\QQ}-\vdual{\Grad\grad v}{\QQ}_\cT.
\end{align*}
Note that $\Grad\grad v$ is the Hessian matrix of $v$.

We collect the needed relations in the following lemma. Its proof is straightforward.

\begin{lemma} \label{la_Stokes_tech}
Let $\omega\subset\Omega$ be a Lipschitz domain. The following relations hold true,
\begin{alignat*}{2}
   &\vdual{\QQ}{\Grad\curl v}_\omega = -\vdual{\QQ^\perp}{\Grad\grad v}_\omega
   &&\quad\forall\QQ\in\LL_2(\omega),\ v\in H^2(\omega),\\
   &\rot\Div\QQ=\div\Div\QQ^\perp        &&\quad\forall \QQ\in\HrotDiv{\omega},\\
   &\dual{\tr{}(v)}{\QQ}_\cS = \dual{\traceGG{}(v)}{\QQ^\perp}_\cS
   &&\quad\forall v\in H^2(\Omega),\ \QQ\in\HrotDiv{\cT},\\
   &\dual{\tra{}(\QQ)}{v}_\cS = \dual{\traceDD{}(\QQ^\perp)}{v}_\cS
   &&\quad\forall \QQ\in\HrotDiv{\Omega},\ v\in H^2(\cT).
\end{alignat*}
In particular, we have the identity/equivalence of spaces
\[
   H^2(\omega)=\HGradcurl{\omega}, \qquad
   \QQ\in\HrotDiv{\omega}\Leftrightarrow \QQ^\perp\in\HdDiv{\omega}
\]
with identical norms.
\end{lemma}

Using these relations, bilinear form \eqref{b} transforms for Stokes problem \eqref{Stokes}
into
\begin{align} \label{b_Stokes}
   b(u,\PP,\tu,\tP;v,\QQ) &=
   -\vdual{u}{\rot\Div\QQ}_\cT + \vdual{\PP}{\QQ-\Grad\curl v}_\cT
   + \dualGC{\tu}{\QQ}_\cS + \dualGCa{\tP}{v}_\cS\nonumber\\
   &=
   -\vdual{u}{\div\Div\QQ^\perp}_\cT + \vdual{\PP^\perp}{\QQ^\perp+\Grad\grad v}_\cT
   + \dual{\tu^\perp}{\QQ^\perp}_\cS + \dual{\tP^\perp}{v}_\cS\nonumber\\
   &=: b^\perp(u,\PP^\perp,\tu^\perp,\tP^\perp;v,\QQ^\perp)
\end{align}
with
\[
   \tu^\perp:=\traceGG{}(u) \quad\text{and}\quad \tP^\perp:=\traceDD{}(\PP^\perp).
\]
Switching in the following to bilinear form $b^\perp$, the appropriate spaces become
\begin{align*}
   &\UU^\perp:=L_2(\Omega)\times\LL_2^s(\Omega)\times
         \traceGG{}(H^2_0(\Omega))\times \traceDD{}(\HdDiv{\Omega}),
   &\VV^\perp:=H^2(\cT)\times\HdDiv{\cT}
\end{align*}
with respective (squared) norms
\begin{align*}
   &\|(u,\PP^\perp,\tu^\perp,\tP^\perp)\|_{\UU^\perp}^2
   := \|u\|^2 + \|\PP^\perp\|^2 + \|\tu^\perp\|_{\inf,2}^2 + \|\tP^\perp\|_{\inf,\div\Div}^2,\\
   &\|(v,\QQ^\perp)\|_{\VV^\perp}^2 := \|v\|_{2,\cT}^2 + \|\QQ^\perp\|_{\div\Div,\cT}^2.
\end{align*}
Here, $\|v\|_{2,\cT}=(\|v\|^2+\|\Grad\grad v\|_\cT^2)^{1/2}$ is the product norm in $H^2(\cT)$,
and $\|\cdot\|_{\inf,2}$, $\|\cdot\|_{\inf,\div\Div}$
are the canonical minimal extension norms in $H^2_0(\Omega)$ and $\HdDiv{\Omega}$, respectively,
with corresponding norms on $\Omega$.

For reference let us state the corresponding norm identities in the ansatz and test spaces.
The proof is immediate.

\begin{lemma} \label{la_Stokes_norms}
Using the $(\cdot)^\perp$ relation previously introduced, the following identities hold true,
\begin{align*}
   &\|(u,\PP,\tr{}(\tilde u),\tra{}(\widetilde\PP)\|_{\UU}
   = \|(u,\PP^\perp,\traceGG{}(\tilde u),\traceDD{}(\widetilde\PP^\perp)\|_{\UU^\perp}\\
   & \hspace{0.3\textwidth}\forall u\in L_2(\Omega),\ \PP\in\LL_2^\perp(\Omega),\
   \tilde u\in H^2(\Omega),\ \widetilde\PP\in\HrotDiv{\Omega},\\
   &\|(v,\QQ)\|_\VV = \|(v,\QQ^\perp)\|_{\VV^\perp} \quad\forall (v,\QQ)\in \VV.
\end{align*}
\end{lemma}

Now, in order to introduce the DPG scheme solving \eqref{Stokes}, we provide three discretization
steps: the selection of discrete subspaces $\UU_h^\perp\subset\UU^\perp$, an approximation of
datum $\rot\ff\in H^{-1}(\Omega)$, and a discretization of trial-to-test operator
$\ttt^\perp:\;\UU^\perp_h\to\VV^\perp$. The final discrete DPG scheme is formulated
in \S\ref{sec_Stokes_DPG}. There, we also formulate the main results of this section,
Theorem~\ref{thm_DPG_Stokes} and Corollary~\ref{cor_DPG_Stokes}, stating a priori error estimates.
Some more technical results and a proof of Theorem~\ref{thm_DPG_Stokes} are given
in \S\S\ref{sec_Stokes_pf1} and~\ref{sec_Stokes_pf2}.

\subsection{Discretization space}

We employ the construction of discrete trace subspaces from \cite{FuehrerHN_19_UFK}.
We assume the mesh $\cT$ to be a regular triangulation of $\Omega$
with shape-regular elements. We will use the mesh width function $h_\cT=(h_T)_{T\in\cT}$
defined with $h_T:=\diam(T)$.
For an element $T\in\cT$, we denote by $P^k(T)$ the space of polynomials
on $T$ of total degree less than or equal to $k$, and
\[
   P^k(\cT):=\{z\in L_2(\Omega);\; z|_T\in P^k(T)\; \forall T\in\cT\}.
\]
Furthermore,
\[
   \mathbb{P}^{k,s}(T):=P^k(T)^{2\times 2}\cap\LL_2^s(T)\quad (T\in\cT),
\]
and $\mathbb{P}^{k,s}(\cT)$ is the corresponding piecewise polynomial tensor-valued
space.
For the discrete trace spaces we need to introduce polynomial spaces on edges. To this end
let $\cE_T$ denote the set of edges of $T\in\cT$ and define $\cE:=\bigcup_{T\in\cT}\cE_T$.
Then, $P^k(E)$ denotes the space of polynomials on $E\in\cE$ up to degree $k$, and
\[
   P^k(\cE_T):=\{z\in L_2(\partial T);\; z|_E\in P^k(E)\;\forall E\in\cE(T)\}\quad (T\in\cT).
\]
The definition of discrete trace spaces is done through the application of trace operators
to appropriate spaces on $\Omega$. For $T\in\cT$ we define
\begin{align*}
   \widehat U_T &:= \traceGG{T}\bigl(
          \{z\in H^2(T);\; \Delta^2 z+z=0,\;
            z|_{\partial T}\in P^3(\cE_T),\; \nn\cdot\grad z|_{\partial T}\in P^1(\cE_T)\}\bigr)
   \quad (T\in\cT),\\
   \widehat U_\cS
   &:= \{\tu=(\tu_T)_{T\in\cT}\in\traceGG{}(H^2_0(\Omega));\; \tu_T\in\widehat U_T\;\forall T\in\cT\}.
\end{align*}
Here, $\traceGG{T}$ denotes the trace operator $\traceGG{}$ associated with element $T\in\cT$
(obtained by formally replacing $\Omega$ and $\cT$ with $T$ and $\{\Omega\}$, respectively,
in the definition of $\traceGG{}$).
The degrees of freedom for the discrete space $\widehat U_\cS$
are the values of functions and their gradients at the
interior nodes, see~\cite{FuehrerHN_19_UFK} for details. That is, the dimension of
$\widehat U_\cS$ is three times the number of interior nodes.

For the definition of a discrete subspace $\widehat Q_\cS\subset \traceDD{}(\HdDiv{\Omega})$
we select local spaces element-wise,
\begin{align*}
   U_{\dDiv,T}:= &\bigl\{\MM\in\HdDiv{T};\; \Grad\grad\div\Div\MM+\MM=0,\\
      &\quad
      \bigl(\nn\cdot\Div\MM+\partial_{\bt,\cE_T}(\bt\cdot\MM\nn)\bigr)|_{\partial T}\in P^0(\cE_T),\
      \nn\cdot\MM\nn|_{\partial T}\in P^0(\cE_T)\bigr\} && (T\in\cT).
\end{align*}
Here, $\bt$ is the unit tangential vector along $\partial T$ in mathematically positive orientation,
and $\partial_{\bt,\cE_T}$ denotes the $\cE_T$-piecewise tangential derivative operator.
The corresponding global space is
\begin{align*}
   U_{\dDiv,\cT}:= \bigl(\Pi_{T\in\cT} U_{\dDiv,T}\bigr) \cap \HdDiv{\Omega}.
\end{align*}
Our second discrete trace space then is
\begin{equation*}
   \widehat Q_\cS := \traceDD{}\bigl(U_{\dDiv,\cT}\bigr).
\end{equation*}
To identify the degrees of freedom we denote by $\cN_T$ the set of vertices of $T\in\cT$,
and let $\cN_{0}$ be the set of vertices of $\cT$ that do not lie on $\Gamma$.
The degrees of freedom of $\widehat Q_\cS$ are
\begin{subequations} \label{dof_dd}
\begin{align}
   \dual{\nn\cdot\MM\nn}1_E,\quad \dual{\nn\cdot\Div\MM + \partial_{\bt}(\bt\cdot\MM\nn)}1_E\quad
   & (E\in\cE), \\
   \jump\MM_{\partial T}(\bx) \qquad & (\bx\in\cN_T, \, T\in\cT), \label{dof_dd_jump}\\
   \text{subject to } \,\,
   \sum_{T\in\cT(\bx)} \jump{\MM}_{\partial T}(\bx) = 0 \qquad & (\bx\in\cN_0). \label{dof_dd_constraint}
\end{align}
\end{subequations}
Here, term $\jump\MM_{\partial T}$ represents point distributions at vertices.
For a sufficiently smooth function $\MM$, they are given by the jumps
\begin{align*}
  \jump\MM_{\partial T}(\bx) = (\bt\cdot\MM\nn)|_E(\bx)-(\bt\cdot\MM\nn)|_{E'}(\bx)
\end{align*}
where $\bx=\overline E\cap \overline E'$ for the edges $E,E'\subset\partial T$
with $\bx$ being an endpoint of $E$ and starting point of $E'$ (seen in the direction of $\bt$).
In the constraint relation at interior nodes, $\cT(\bx)$ is the set of elements $T\in\cT$
that have node $\bx$ as a vertex.
For details on the degrees of freedom we refer to~\cite[Section~3.4]{FuehrerHN_19_UFK}.
The dimension of $\widehat Q_\cS$ is $2\#\cE+3\#\cT-\#\cN_0$.

Finally, our approximation space is defined as
\begin{equation*}
  \UU^\perp_h := P^1(\cT)\times \mathbb{P}^{0,s}(\cT)\times\widehat U_\cS\times \widehat Q_\cS.
\end{equation*}

\subsection{Data approximation}

Considering Stokes problem \eqref{Stokes}, the linear functional $L$ from variational formulation
\eqref{VF} becomes $L(v,\QQ)=\vdual{\rot\ff}{v}$ and is not well defined for $v\in H^2(\cT)$.
We circumvent this problem by approximating $\rot\ff\in H^{-1}(\Omega)$ by an $L_2(\Omega)$-load.
We do this by using the operator $P_h':\;H^{-1}(\Omega)\to P^1(\cT)$ from \cite{Fuehrer_21_MDN}
and repeating the procedure and analysis from \cite{FuehrerHK_22_MSO}.

Let us give some details. We start with the Scott--Zhang-type quasi-interpolation operator
$J_h:\;L_2(\Omega)\to P^1(\cT)\cap H^1_0(\Omega)$ defined by
\[
   J_hv=\sum_{\bx\in\cN_0} \vdual{v}{\psi_{\bx}}\eta_{\bx}.
\]
Here, $\eta_{\bx}\in P^1(\cT)\cap H^1_0(\Omega)$ is the nodal basis function associated with
$\bx$, normed by $\|\eta_{\bx}\|_\infty=1$, and $\psi_{\bx}\in P^1(\cT)$ is bi-orthogonal to $\eta_{\bx}$
with support on the closure of the subdomain $\Omega(\bx)\subset\Omega$
generated by the elements surrounding vertex $\bx$,
\(
   \overline\Omega(\bx):=\cup\{\overline T\in\cT;\; \bx\in\cN_T\}.
\)
Furthermore, we introduce an operator $B_h:\;L_2(\Omega)\to L_2(\Omega)$ by
\[
   B_hv :=\sum_{T\in\cT} \vdual{v}{\chi_T}_T \eta_{b,T}
\]
with characteristic function $\chi_T$ on $T$ and element bubble function
$\eta_{b,T}\in P^3(T)\cap H^1_0(T)$ ($T\in\cT$), normalized as
$\vdual{\eta_{b,T}}1_T=1$ and extended by $0$ to $\Omega$. Finally we define operator
$P_h':\;H^{-1}(\Omega)\to P^1(\cT)$ as the adjoint of $J_h(\cdot)+B_h(1-J_h)(\cdot)$,
\[
   P_h'\phi:=J_h'\phi+(1-J_h')B_h'\phi\quad\text{with}\quad
   J_h'\phi=\sum_{\bx\in\cN_0}\vdual{\phi}{\eta_{\bx}}\psi_{\bx},\quad
   B_h'\phi=\sum_{T\in\cT}\vdual{\phi}{\eta_{b,T}}\chi_T.
\]
For further details and properties we refer to \cite{Fuehrer_21_MDN,FuehrerHK_22_MSO}.

We use operator $P_h'$ to regularize the $H^{-1}(\Omega)$-functional of the Stokes problem,
\begin{equation*}
   L_h(v,\QQ^\perp):=\vdual{P_h'\rot\ff}{v}\quad\forall (v,\QQ^\perp)\in\VV^\perp.
\end{equation*}

\subsection{Fully discrete scheme and a priori error estimate} \label{sec_Stokes_DPG}

In practice, trial-to-test operator $\ttt:\;\UU\to\VV$ has to be approximated. The standard
procedure is to replace $\VV$ with a piecewise polynomial space $\VV_h$ of polynomial degrees
that are sufficiently large. Critical is the existence of a Fortin operator
$\Pi_F:\;\VV\to\VV_h$ which is uniformly bounded, and such that
$b(\ww,(1-\Pi_F)\vv)=0$ for any $\ww\in\UU_h$ and $\vv\in\VV$, cf.~\cite{GopalakrishnanQ_14_APD}.
Considering the transformed bilinear form $b^\perp$, we can use the construction
from \cite{FuehrerH_19_FDD} for the Kirchhoff--Love model problem, see Section~2.5 there.
This leads us to selecting
\begin{equation*}
   \VV_h^\perp:=P^3(\cT)\times\mathbb{P}^{4,s}(\cT)
\end{equation*}
and defining $\ttt_h^\perp:\;\UU_h^\perp\to \VV_h^\perp$ by
\[
   \ip{\ttt_h^\perp(\ww^\perp)}{\dvv^\perp}_{\VV^\perp} = b^\perp(\ww^\perp,\dvv^\perp)
   \quad\forall \ww^\perp\in\UU_h^\perp,\ \dvv^\perp\in\VV_h^\perp.
\]
Here, $\ip{\cdot}{\cdot}_{\VV^\perp}$ is the inner product in $\VV^\perp$ that gives rise
to norm $\|\cdot\|_{\VV^\perp}$ defined before.
With this preparation, our fully discrete DPG scheme for Stokes problem \eqref{Stokes} reads as
\begin{equation} \label{DPG_Stokes}
   \uu^\perp_h:=(u_h,\PP^\perp_h,\tu^\perp_h,\tP^\perp_h)\in\UU^\perp_h:\quad
   b^\perp(\uu^\perp_h,\dvv^\perp) = L_h(\dvv^\perp)
   \quad\forall \dvv^\perp\in\ttt_h^\perp(\UU_h^\perp).
\end{equation}
In order to specify the approximation order of the lowest-order DPG scheme,
let us introduce the regularity shift $s=s(\Omega)$ of the Laplacian:
\begin{equation} \label{shift}
   s\in (0,1]:\; (-\Delta)^{-1}:\; L_2(\Omega)\to H^{1+s}(\Omega)\cap H^1_0(\Omega).
\end{equation}
It goes without saying that we are interested in $s$ as large as possible. Though a maximum
value need not exist.
Furthermore, we make the assumption that the solution $u$ of \eqref{Stokes} satisfies
the regularity property
\begin{equation} \label{reg_Stokes}
   \|u\|_3\lesssim \|\rot\ff\|_{-1}.
\end{equation}
It follows from the more general assumption that given $f\in H^{-1}(\Omega)$ the solution of the bi-Poisson problem
  \begin{align*}
    \Delta^2 w = f \text{ in }\Omega, \quad w=0=\partial_{\bn}w \text{ on }\Gamma,
  \end{align*}
  satisfies the regularity property
  \begin{align}\label{reg_biharmonic}
    \|w\|_3\lesssim \|f\|_{-1}.
  \end{align}
This is certainly true for a convex domain $\Omega$, cf.~\cite[Theorem~2]{BlumR_80_BVP}.
Of course, in the case of a convex domain, the maximum regularity shift \eqref{shift} of the Laplacian
is $s=1$. However, some of the technical results we prove below apply to non-convex domains
and are relevant for plate bending problems. Therefore, we keep relations \eqref{shift}
and \eqref{reg_Stokes} independent.

\begin{theorem} \label{thm_DPG_Stokes}
We consider a regular polygonal domain $\Omega\subset\R^2$.
Let $\ff\in \bL_2(\Omega)$ be given and let $u$ be the solution of \eqref{Stokes}.
We assume that \eqref{reg_Stokes} holds, set $\PP:=-\Grad\curl u$, $\tu:=\tr{}(u)$,
and recall that $s>0$ is the regularity shift \eqref{shift}.

(i) Scheme \eqref{DPG_Stokes} has a unique solution $\uu_h^\perp\in\UU_h^\perp$
with corresponding function $\uu_h=(u_h,\PP_h,\tu_h,\tP_h)\in\UU_h$. It satisfies
\begin{align*}
    \|u-u_h\| + \|\PP-\PP_h\| + \|\tu-\tu_h\|_{\inf}
    \lesssim \|h_\cT\|_\infty^s \|\ff\|
\end{align*}
with a hidden constant that is independent of $\cT$.

(ii) Assume additionally that $\rot\ff\in L_2(\Omega)$.
Then, the solution component $u_h$ satisfies the improved error estimate
  \begin{align*}
    \|u-u_h\| \lesssim \|h_\cT\|_\infty^{1+s} (\|\ff\|+\|\rot\ff\|)
  \end{align*}
with a hidden constant that is independent of $\cT$.
\end{theorem}

The proof of these statements is quite technical and requires some preparation.
For part (i) this is done in \S\ref{sec_Stokes_pf1}, and for part (ii) in \S\ref{sec_Stokes_pf2}.
Before giving a proof, let us note that Theorem~\ref{thm_DPG_Stokes}(ii) implies
an error estimate for the approximation of the velocity field $\bu=\curl u$ by $\curl_\cT u_h$.
A proof is direct and therefore skipped.

\begin{cor}\label{cor_DPG_Stokes}
  Under the assumptions and notation of Theorem~\ref{thm_DPG_Stokes}(ii),
  the discrete velocity field $\bu_h:=\curl_\cT u_h$ satisfies
  \begin{align*}
    \|\bu-\bu_h\| \lesssim \|h_\cT\|_\infty^{s} (\|\ff\|+\|\rot\ff\|)
  \end{align*}
with a hidden constant that is independent of $\cT$.
\end{cor}

\subsection{Approximation results and proof of Theorem~\ref{thm_DPG_Stokes}(i)} \label{sec_Stokes_pf1}

For ease of presentation we abbreviate $h:=\|h_\cT\|_\infty$ in the proofs of this
and the following section.

We start by proving a proposition that improves the trace approximation result from
\cite[Lemma~6.4]{FuehrerHN_19_UFK}.
We need the lowest-order Raviart--Thomas interpolation operator associated with
mesh $\cT$, $\Pi^0_{\div}$, and the corresponding operator $\Pi^0_{\Div}$ applied row-wise to tensors.
Furthermore, we employ the $L_2$-projector $\Pi^0$ onto piecewise constants. It is generically used for
scalar and vector functions. We will use the properties
$\Div\Pi^0_{\div}\MM=\Pi^0\Div\MM$ and $\|\MM-\Pi^0_{\div}\MM\|\lesssim h\|\MM\|_1$
for $\MM\in\HH^1(\Omega)$, and will use this notation and the mentioned relations also element-wise. 

\begin{prop} \label{prop_tP_approx}
Let $\Omega$ be a regular polygonal domain with regularity shift $s$, cf.~\eqref{shift}.
Given $\MM\in\HdDiv{\Omega}\cap \HH^1(\Omega)$, estimate
\begin{align*}
   \min_{\tq_h \in \widehat Q_\cS}
   \|\traceDD{}(\MM) - \tq_h\|_{\inf,\div\Div}
   \lesssim \|h_\cT\|_\infty \|\MM\|_1 + \|h_\cT\|_\infty^{1+s} \|\div\Div\MM\|
\end{align*}
holds true with a hidden constant that only depends on the shape-regularity of $\cT$.
\end{prop}

\begin{proof}
We will use that $\|\tq\|_{\inf,\div\Div}=\sup_{\|z\|_{2,\cT}=1}\dual{\tq}{z}_\cS$
by \cite[Proposition~3.5]{FuehrerHN_19_UFK}.
Let $\MM\in\HdDiv{\Omega}\cap\HH^1(\Omega)$ be given and denote $\tq:=\traceDD{}(\MM)$.
By construction of $\widehat Q_\cS$, it is sufficient to prove the existence of
$\tq_h \in \traceDD{}(U_{\dDiv,\cT})$ such that 
\begin{align*}
    |\dual{\tq-\tq_h}{z}_{\cS}|
   \lesssim \bigl(h \|\MM\|_1 + h^{1+s} \|\div\Div\MM\|\bigr) \|z\|_{2,\cT}
   \quad\forall z\in H^2(\cT).
\end{align*}
We split the proof into three steps.
  
\textbf{Step 1: decomposition of $\MM$.}
We use the decomposition of $\HdDiv{\Omega}$-tensor fields
from~\cite[Theorem~4.2]{RafetsederZ_18_DRK}:
there exist $\eta\in H_0^1(\Omega)$, $\eeta\in H^1(\Omega)^2$ such that
\begin{align*}
    \MM = \MM_1 + \MM_2 := \eta\II + \scurl \eeta
\end{align*}
where $\Delta \eta = \div\Div\MM \in L^2(\Omega)$, and
$\scurl\eeta = \big(\curl\eeta+(\curl\eeta)^\top\big)/2$ is the symmetrized $\curl$,
the latter being defined component-wise for vector functions as
$\curl\begin{pmatrix}\eta_1\\ \eta_2\end{pmatrix}
:=\begin{pmatrix} \curl\eta_1^\top\\ \curl\eta_2^\top\end{pmatrix}$.
Note that $\Div\MM_1=\grad\eta\in\bL_2(\Omega)$, $\Div\MM_2=\Div(\MM-\MM_1)\in\bL_2(\Omega)$,
$\div\Div\MM_1=\div\Div\MM\in L_2(\Omega)$, and $\div\Div\MM_2=0$, so that
$\MM_1$, $\MM_2\in\HdDiv{\Omega}\cap\HDiv{\Omega}$ and $\MM_2\in\HH^1(\Omega)$ with
\begin{align*}
   &\|\Div\MM_1\| = \|\grad \eta\| = \|\div\Div\MM\|_{-1} \lesssim \|\Div\MM\|, &&
   \|\div\Div\MM_1\| = \|\div\Div\MM\|,\\
   &\|\Div\MM_2\| = \|\Div\MM-\grad\eta\| \lesssim \|\Div\MM\|, &&
   \|\div\Div\MM_2\| = 0,\quad \|\MM_2\|_1 \lesssim \|\MM\|_1.
\end{align*}
Here, we used the boundedness of $\div:\;\bL_2(\Omega)\to H^{-1}(\Omega)$ and
relation $\|\grad\eta\|=\|\div\Div\MM\|_{-1}$.
Furthermore, by definition of $s$, $\eta\in H^{1+s}(\Omega)$ with
corresponding bound
\[
   \|\eta\|_{1+s}\lesssim \|\div\Div\MM\|.
\]
Now, trace $\tq:=\traceDD{}(\MM)$ has the induced decomposition
\[
   \tq=\tq_1+\tq_2:= \traceDD{}(\MM_1)+\traceDD{}(\MM_2),
\]
and in what follows we construct corresponding approximations $\tq_{h,1}$, $\tq_{h,2}\in\widehat Q_\cS$
and choose $\tq_h:=\tq_{h,1}+\tq_{h,2}$ as the approximation of $\tq$.
Since $\MM_j\in\HdDiv{\Omega}\cap\HH^1(\Omega)$, traces $\tq_j$ have canonical components
\begin{align} \label{tq_0}
   &\dual{\tq_j}{z}_{\cS} = \dual{\traceDD{}(\MM_j)}{z}_{\cS}
   = \sum_{T\in\cT} \dual{\nn\cdot\Div\MM_j}{z}_{L_2(\partial T)}
                  - \dual{\MM_j\nn}{\grad z}_{L_2(\partial T)}
\end{align}
for $z\in H^2(\cT)$, $j=1,2$, cf.~\cite[Remark~3.1]{FuehrerHN_19_UFK}.
Here, $\dual{\cdot}{\cdot}_{L_2(\partial T)}$ denotes the appropriate duality pairing
on $\partial T$ with $L_2(\partial T)$ as pivot space, cf.~Remark~\ref{rem_trace}.
In order to construct the trace approximations we need to localize the trace components
with respect to edges. We note that $\MM_2$ is not sufficiently regular to do this directly,
whereas $\MM_1$ does have the required regularity.

\textbf{Step 2 (construction of $\tq_{h,1}$).}
We argue as in the proof of~\cite[Lemma~6.4]{FuehrerHN_19_UFK}, and localize \eqref{tq_0} as follows,
\begin{align}
   &\dual{\tq_1}{z}_{\partial T} \nonumber\\
   &= \sum_{E\in\cE_T}
   \Big(\dual{\nn\cdot\Div\MM_1+\partial_{\bt}(\bt\cdot\MM_1\nn)}{z}_E
    - \dual{\nn\cdot\MM_1\nn}{\nn\cdot\grad z}_E\Bigr)
    -\sum_{\bx\in \cN_T} \jump{\MM_1}_{\partial T}(\bx)z(\bx) \label{tq1a}\\
   &= \sum_{E\in\cE_T}
   \Big(\dual{\nn\cdot\grad\eta}{z}_E
    - \dual{\eta}{\nn\cdot\grad z}_E\Bigr) \quad\forall z\in H^2(T),\ T\in\cT.
\label{tq1}
\end{align}
Here, $\partial_{\bt}(\cdot)$ is the edge-wise tangential derivative,
and occasionally we will write $\nn\cdot\grad z=\partial_{\nn} z$.
Indeed, all the traces on $E\in\cE_T$ are $L_2(E)$-regular.
In this case, we used that $\bt\cdot\MM_1\nn=\bt\cdot\II\nn\eta=0$
on $E\in\cE_T$ so that $\partial_{\bt}(\bt\cdot\MM_1\nn)|_E=0$ and $\jump{\MM_1}_{\partial T}(\bx)=0$
for every edge $E$ and vertex $\bx$ of $T$.

Now, since $\partial_{\nn}\eta|_E\in L^2(E)$ for every edge,
there exist antiderivatives $g_E\in H^1(E)$ with
$\partial_{\bt} g_E|_E = \partial_{\nn}\eta|_E$ ($E\in\cE$).
For $T\in\cT$ we define $g_{\cE_T}^0\in H^1(\cE_T)$
(edge-wise $H^1$-functions) as $g_{\cE_T}^0|_E = g_{E}$ $\forall E\in\cE_T$.
We then choose the following degrees of freedom~\eqref{dof_dd} for $\tq_{h,1}$,
\begin{align*}
   \dual{\Pi^0_E\eta}1_E,\quad \dual{\partial_{\bt} \Pi_E^1 g_E}1_E
   &\quad (E\in\cE),\quad
   \jump{(1-\Pi_{\cE_T}^1) g_{\cE_T}^0}_{\partial T}(\bx) \quad(\bx\in\cN_T,\ T\in\cT).
\end{align*}
Here, $\Pi_E^k:\; L^2(E)\to P^k(E)$ denotes the $L^2(E)$-projector,
$\Pi_{\cE_T}^1:\; L^2(\partial T)\to P^1(\cE_T)$ is defined as
$\Pi_{\cE_T}^1 u|_E = \Pi_E^1 (u|_E)$ for $E\in\cE$, and 
$\jump{\cdot}_{\partial T}(\bx)$ is used as the jump at $\bx$ also for scalar functions
with the sign convention as in \eqref{dof_dd_jump}.
By construction, constraint \eqref{dof_dd_constraint} is satisfied.

Let $T\in\cT$ be given. Recalling relation \eqref{tq1} and integrating by parts on edges,
we calculate for $z\in H^2(T)$
\begin{align}
   &\dual{\tq_1-\tq_{h,1}}{z}_{\partial T}
   =
   \sum_{E\in\cE_T}
   \Big(\dual{\partial_{\bt}g_E}{z}_E - \dual{\eta}{\nn\cdot\grad z}_E\Bigr) \nonumber\\
   &\hspace{7em}
   - \sum_{E\in\cE_T}
   \bigl(\dual{\partial_{\bt}\Pi^1_E g_E}{z}_E - \dual{\Pi^0_E\eta}{\nn\cdot\grad z}_E\bigr)
   + \sum_{\bx\in\cN_T} \jump{(1-\Pi_{\cE_T}^{1}) g_{\cE_T}}_{\partial T}(\bx)z(\bx) \nonumber\\
   &=
   - \sum_{E\in\cE_T}
   \bigl(\dual{(1-\Pi^1_E) g_E}{\partial_{\bt}z}_E + \dual{(1-\Pi^0_E)\eta}{\nn\cdot\grad z}_E\bigr).
   \label{term2}
\end{align}
The last term in \eqref{term2} is estimated as
\begin{align} \label{term2b}
   |\dual{(1-\Pi_E^0)\eta}{\partial_{\nn} z}_E|
   = |\dual{(1-\Pi_E^0)\eta}{(1-\Pi_E^0)\partial_{\nn} z}_E| \lesssim h_T \|\eta\|_{1,T}\|z\|_{2,T},
\end{align}
and the first term in \eqref{term2} is bounded as
\begin{align} \label{term2c}
   |\dual{(1-\Pi_E^1)g_E}{\partial_{\bt} z}_E|
   &= |\dual{(1-\Pi_E^1)g_E}{(1-\Pi_E^1)\partial_{\bt} z}_E|
   \lesssim h_T\|(1-\Pi_E^0)\partial_{\bt} g_E\|_E h_T^{1/2}\|z\|_{2,T} \nonumber\\
   &= h_T^{3/2} \|(1-\Pi_E^0)\partial_{\nn}\eta\|_E \|z\|_{2,T}
   \lesssim h_T^{1+s} \|\eta\|_{1+s,T} \|z\|_{2,T}.
\end{align}
Here, bound $\|(1-\Pi_E^1)g_E\|_E\lesssim h_T\|(1-\Pi_E^0)\partial_{\bt} g_E\|_E$
follows from a Bramble--Hilbert argument and
$\|(1-\Pi_E^0)\partial_{\nn}\eta\|_E\lesssim h_T^{s-1/2}\|\eta\|_{1+s,T}$ is due to a trace
argument, cf.~\cite{Grisvard_85_EPN}, and scaling properties, cf.~\cite{Heuer_14_OEF}.
Combining relation \eqref{term2} with bounds \eqref{term2b}, \eqref{term2c}, and summing
over all elements $T\in\cT$, we conclude that
\begin{align} \label{est_11}
   |\dual{\tq_1-\tq_{h,1}}{z}_{\cS}|
   &\lesssim
   \big( h \|\eta\|_1 + h^{1+s} \|\eta\|_{1+s} \big) \|z\|_{2,\cT}
   \lesssim
   \bigl(h \|\Div \MM\| + h^{1+s} \|\div\Div\MM\| \bigr) \|z\|_{2,\cT}
\end{align}
holds for any $z\in H^2(\cT)$.
In the first estimate we applied the inequality $\|\cdot\|_{1+s,\cT}\lesssim \|\cdot\|_{1+s}$
which is immediate for the Sobolev--Slobodeckij norm, see
relations \cite[(3.6), (3.7)]{BespalovH_08_hpB} which also apply to orders $1+s\in(1,2)$.
We refer to \cite{Faehrmann_00_LAS} for related estimates.
We note that the particular choice of a norm in the spaces
$H^{1+s}(T)$ ($T\in\cT$) is not relevant for shape-regular elements, see \cite{Heuer_14_OEF}.
For the second inequality in \eqref{est_11} we used the regularity results from Step~1.

\textbf{Step 3 (construction of $\tq_{h,2}$).}
As mentioned before, in this case we cannot proceed as in \eqref{tq1}. In fact, in general
$\nn\cdot\Div\MM_2|_{\partial T}\in H^{-1/2}(\partial T)$ which is not localizable on edges to be tested
with traces of $H^2(T)$-functions.
However, since $\div\Div\MM_2=0$, there exists $\widetilde\eta\in H^1(\Omega)$
with $\Div\MM_2 = \curl \widetilde\eta$,
$\|\widetilde\eta\|_1 \lesssim \|\Div\MM_2\|$ and
$\nn\cdot\Div\MM_2 = \nn\cdot\curl\widetilde\eta$ on $\partial T$ for any $T\in\cT$.
Then, proceeding as in \eqref{tq1} and integrating by parts, we obtain
\begin{align} \label{tq2}
   \dual{\tq_2}{z}_{\partial T}
   &= \dual{\nn\cdot\Div\MM_2}{z}_{\partial T}
   - \dual{\MM_2\nn}{\grad z}_{\partial T}
   \nonumber\\
   &=\dual{\partial_{\bt}\widetilde\eta}{z}_{\partial T}
   - \dual{\bt\cdot\MM_2\nn}{\partial_{\bt}z}_{\partial T}
   - \dual{\nn\cdot\MM_2\nn}{\partial_{\nn} z}_{\partial T} 
   \nonumber\\
   &= -\sum_{E\in\cE_T}\bigl(
      \dual{\widetilde\eta+\bt\cdot\scurl\eeta\nn}{\partial_{\bt} z}_E
      + \dual{\nn\cdot\scurl\eeta\nn}{\partial_{\nn} z}_E\bigr)
   \quad\forall z\in H^2(T).
\end{align}
Note that the latter localization is possible since we test with $z\in H^2(T)$ so that
$\partial_{\bt} z|_{\partial T}, \partial_{\nn} z|_{\partial T}\in L^2(\partial T)$.
Representation \eqref{tq2} leads us to choose the following degrees of freedom
\eqref{dof_dd} for $\tq_{h,2}$,
\begin{align*}
   \Pi_E^0 (\nn\cdot\scurl \eeta \nn)|_E &\quad (E\in\cE), \\
   \partial_{\bt} \Pi_E^1 \widetilde\eta|_E
   + \partial_{\bt} \Pi_E^1(\bt\cdot\scurl\eeta \nn)|_E &\quad (E\in\cE), \\
   \jump{\scurl\eeta}_{\partial T}(\bx)
   - \jump{\Pi_{\cE_T}^1\widetilde\eta}_{\partial T}(\bx)
   - \jump{\Pi_{\cE_T}^1\scurl\eeta}_{\partial T}(\bx) &\quad (\bx\in\cN).
\end{align*}
Note that $\partial_{\bt}p_1|_E\in P^0(E)$ if $p_1|_E \in P^1(E)$.
Similarly as in Step~2 we abbreviate as $\jump{\Pi_{\cE_T}^1\widetilde\eta}_{\partial T}(\bx)$
and $\jump{\Pi_{\cE_T}^1\scurl\eeta}_{\partial T}(\bx)$
the corresponding jumps across two edges of an element $T$ with the sign convention as described before,
and using the same sign convention as for $\jump{\scurl\eeta}_{\partial T}(\bx)$.
Note that the latter choices satisfy~\eqref{dof_dd_constraint}.
  
Using \eqref{tq2}, and integrating edge-wise by parts, we obtain for any $z\in H^2(T)$ and $T\in\cT$
\begin{align*}
   &\dual{\tq_2-\tq_{h,2}}{z}_{\partial T}
   = -\dual{(1-\Pi_{\cE_T}^1)(\widetilde\eta+\bt\cdot\scurl\eeta\nn)}{\partial_{\bt} z}_{\partial T}
   - \dual{(1-\Pi_{\cE_T}^0)\nn\cdot\scurl\eeta\nn}{\partial_{\nn} z}_{\partial T}.
\end{align*}
Now, the projection property of $\Pi_{\cE_T}^p$ and the trace approximation inequality
$\|(1-\Pi_E^p)w\|_E \lesssim h^{1/2}\|w\|_{1,T}$ for any $w\in H^1(T)$ show that
\begin{align*}
   &\dual{(1-\Pi_{\cE_T}^1)(\widetilde\eta+\bt\cdot\scurl\eeta\nn)}{\partial_{\bt} z}_{\partial T}
   + \dual{(1-\Pi_{\cE_T}^0)\nn\cdot\scurl\eeta\nn}{\partial_{\nn} z}_{\partial T}
   \\
   &\
   = \dual{(1-\Pi_{\cE_T}^1)(\widetilde\eta+\bt\cdot\scurl\eeta\nn)}
          {(1-\Pi_{\cE_T}^1)\partial_{\bt} z}_{\partial T} 
   + \dual{(1-\Pi_{\cE_T}^0)\nn\cdot\scurl\eeta\nn}{(1-\Pi_{\cE_T}^0)\partial_{\nn} z}_{\partial T}
   \\
   &\
   \lesssim h \bigl(\|\widetilde\eta\|_{1,T} + \|\scurl\eeta\|_{1,T}\bigr)\|z\|_{2,T}
   \quad\forall z\in H^2(T),\ T\in\cT.
\end{align*}
Summing over all elements $T\in\cT$ and using the regularity estimates from Step~1,
we conclude that
\begin{align*}
  |\dual{\tq_2-\tq_{h,2}}z_{\cS}| \lesssim h\|\MM\|_1\|z\|_{2,\cT}\quad\forall z\in H^2(\cT).
\end{align*}
Combining this bound with estimate \eqref{est_11} and  the splitting from Step~1,
we obtain the stated error bound of Proposition~\ref{prop_tP_approx}.
\end{proof}

Before analyzing DPG scheme \eqref{DPG_Stokes} with regularized functional $L_h$ we need to
derive an error estimate for the scheme without regularized functional and datum
$f\in L_2(\Omega)$, replacing $\rot\ff$.
This scheme reads
\begin{equation} \label{DPG_Stokes_L2}
   \uu^\perp_h:=(u_h,\PP^\perp_h,\tu^\perp_h,\tP^\perp_h)\in\UU^\perp_h:\quad
   b^\perp(\uu^\perp_h,\dvv^\perp) = \vdual{f}{\deltav}
   \quad\forall \dvv^\perp=(\deltav,\deltaQQ^\perp)\in\ttt_h^\perp(\UU_h^\perp).
\end{equation}
It provides an approximation to the following variational problem:
\begin{equation} \label{VF_Stokes_L2}
   \uu^\perp\in\UU^\perp:\quad
   b^\perp(\uu^\perp,\dvv^\perp) = \vdual{f}{\deltav}
   \quad\forall \dvv^\perp=(\deltav,\deltaQQ^\perp)\in\VV^\perp.
\end{equation}

\begin{prop} \label{prop_DPG_Stokes_L2}
Let $f\in L_2(\Omega)$ be given. Problem \eqref{VF_Stokes_L2} has a unique solution
$\uu=(u,\PP^\perp,\tu^\perp,\tP^\perp)\in\UU^\perp$. Component $u$ solves
\eqref{Stokes} for a datum $\ff$ that satisfies $\rot\ff=f$.
Furthermore, \eqref{DPG_Stokes_L2} has a unique solution $\uu_h^\perp\in\UU_h^\perp$
with corresponding function $\uu_h\in\UU_h$. Assuming that $u\in H^3(\Omega)$,
and $\Omega$ being a regular polygonal domain, $\uu_h$ satisfies
\begin{align*}
    \|\uu-\uu_h\|_U \lesssim \|h_\cT\|_\infty \|u\|_3 + \|h_\cT\|_\infty^{1+s} \|f\|
\end{align*}
with a hidden constant that is independent of $\cT$. Here, $s>0$ is the regularity shift \eqref{shift}.
\end{prop}

\begin{proof}
We start by noting that discrete system \eqref{DPG_Stokes_L2} inherits the well-posedness and
quasi-optimal convergence from the corresponding scheme with optimal test functions
$\dvv^\perp\in\ttt^\perp(\UU_h^\perp)$ once the existence of a Fortin operator
$\Pi_F:\;\VV^\perp\to\VV_h^\perp$ is guaranteed, see~\cite{GopalakrishnanQ_14_APD}.
Such an operator with the required properties has been constructed in \cite{FuehrerH_19_FDD},
see \S{2.5} there.
Therefore, using relation \eqref{b_Stokes}, Theorem~\ref{thm} shows that
\eqref{VF_Stokes_L2} and \eqref{DPG_Stokes_L2} have unique solutions $\uu^\perp\in\UU^\perp$
and $\uu_h^\perp\in\UU_h^\perp$, respectively.
These solutions correspond to unique elements $\uu\in\UU$ and $\uu_h\in\UU_h$.
Solution $\uu=(u,\PP,\tu,\tP)\in\UU$ satisfies $\PP=-\Grad\curl u$, $\tu=\tr{}(u)$, $\tP=\tra{}(\PP)$,
and $\rot\Div\PP=f$, and discretization
$\uu_h\in\UU_h$ is a quasi best-approximation of $\uu$ in the $\UU$-norm.
In~\cite[Section~6]{FuehrerHN_19_UFK} it is shown that
\begin{align*}
    \min_{(w_h,\SS_h,\tbw_h,0)\in \UU_h^\perp}
    \big(\|u-w_h\| + \|\PP^\perp-\SS_h\| + \|\tu^\perp - \tbw_h\|_{\inf,2}\big)
    \lesssim h \|u\|_3.
\end{align*}
We also have the approximation bound
\begin{align*}
   \min_{\tq_h\in \widehat Q_\cS} \|\tP^\perp - \tq_h\|_{\inf,\div\Div}
   \lesssim h \|\PP^\perp\|_1 + h^{1+s} \|\div\Div\PP^\perp\|
\end{align*} 
by Proposition~\ref{prop_tP_approx}.
The proof of Proposition~\ref{prop_DPG_Stokes_L2} is finished by bounding
$\|\PP^\perp\|_1=\|\Grad\curl u\|_1\lesssim \|u\|_3$, recalling that
$\div\Div\PP^\perp=\rot\Div\PP=f$, cf.~Lemma~\ref{la_Stokes_tech}, and
making use of norm identity $\|\uu\|_\UU=\|\uu^\perp\|_{\UU^\perp}$ for any $\uu\in\UU$
by Lemma~\ref{la_Stokes_norms}.
\end{proof}

{\bf Proof of Theorem~\ref{thm_DPG_Stokes}.}
We follow the proof of \cite[Theorem~11]{FuehrerHK_22_MSO}.
Let $\tilde u$ denote the solution of Stokes problem \eqref{Stokes} with
$\rot\ff$ replaced with $P_h'\rot\ff$,
and set $\tilde\uu:=(\tilde u,\widetilde\PP,\tr{}(\tilde u),\tra{}(\widetilde\PP))$
with $\widetilde\PP:=-\Grad\curl\tilde u$.
The corresponding function
$\tilde\uu^\perp=(\tilde u,\widetilde\PP^\perp,\traceGG{}(\tilde u),\traceDD{}(\widetilde\PP^\perp))$
solves \eqref{VF_Stokes_L2}, see~Proposition~\ref{prop_DPG_Stokes_L2} and~Lemma~\ref{la_Stokes_tech}.
A standard variational formulation of \eqref{Stokes}, the bi-Laplacian, shows that
\(
   \|u-\tilde u\|_2 \simeq \|\rot\ff-P_h'\rot\ff\|_{-2}
\)
with $\|\cdot\|_{-2}$ denoting the canonical norm in the dual space of $H^2_0(\Omega)$.
We then use \cite[Lemma~10]{FuehrerHK_22_MSO} and
the boundedness of $\rot:\;\bL_2(\Omega)\to H^{-1}(\Omega)$ to bound
\[
  \|u-\tilde u\|_2 \lesssim  h \min\{\|\rot\ff-v_h\|_{-1};\; v_h\in P^0(\cT)\}
   \le h \|\rot\ff\|_{-1} \lesssim h \|\ff\|.
\]
This bound implies
\[
   \|u-\tilde u\| + \|\PP-\widetilde\PP\| + \|\tu-\tr{}(\tilde u)\|_{\inf}
   \lesssim  h \|\ff\|.
\]
By Proposition~\ref{prop_DPG_Stokes_L2}, assumption \eqref{reg_Stokes} and an inverse estimate,
we bound
\begin{align*}
    \|\tilde \uu-\uu_h\|_U \lesssim h \|\tilde u\|_3 + h^{1+s} \|P_h'\rot\ff\|
                           \lesssim h \|P_h'\rot\ff\|_{-1} + h^s \|P_h'\rot\ff\|_{-1}
                           \lesssim h^{s} \|P_h'\rot\ff\|_{-1}.
\end{align*}
Using the boundedness of $P_h':\; H^{-1}(\Omega)\to H^{-1}(\Omega)$ by \cite[Lemma~7]{Fuehrer_21_MDN}
and $\rot:\;\bL_2(\Omega)\to H^{-1}(\Omega)$, we conclude that
\(
    \|\tilde \uu-\uu_h\|_U \lesssim h^s \|\ff\|. 
\)
An application of the triangle inequality finishes the proof of Theorem~\ref{thm_DPG_Stokes}. \qed

\subsection{Duality estimates and proof of Theorem~\ref{thm_DPG_Stokes}(ii)} \label{sec_Stokes_pf2}

To show part (ii) of Theorem~\ref{thm_DPG_Stokes} we follow the techniques developed
in~\cite{Fuehrer_19_SDM}.

\begin{prop}\label{prop_dual}
  For a given $g\in L_2(\Omega)$ let $(v,\bQ)\in H_0^2(\Omega)\times \HdDiv{\Omega}$ denote the unique solution of
  \begin{align}\label{eq_dual1}
      -\dDiv \bQ = g,\quad \Grad\grad v+\bQ = 0,\quad
      v|_\Gamma &= 0 = \partial_{\bn}v|_\Gamma.
  \end{align}
  Furthermore, let $(w^*,\bR^*)\in H_0^2(\Omega)\times\HdDiv{\Omega}$ denote the unique solution of
  \begin{align}\label{eq_dual2}
      \dDiv\bR^* = v+g,\quad \bR^*-\Grad\grad w^* &= \bQ,\quad
      w^*|_{\Gamma} = 0 = \partial_{\bn} w^*|_{\Gamma},
  \end{align}
  and set $\ww = (g,0,0,\traceDD{}\bQ) + \ww^* = (g,0,0,\traceDD{}\bQ) + (w^*,\bR^*,\traceGG{\cT}w^*,\traceDD{\cT}\bR^*)\in \UU^\perp$.

  The two functions $\vv=(v,\bQ)$ and $\ww$ are related by $\ttt^\perp\ww=\vv$, i.e., 
  \begin{align}\label{eq_ttt_representation}
    \ip{\vv}{\delta\vv}_{\VV^\perp} = b^\perp(\ww,\delta\vv) \quad\forall \delta\vv \in\VV^\perp.
  \end{align}
\end{prop}
\begin{proof}
  We stress that~\eqref{eq_dual1} and~\eqref{eq_dual2} have unique solutions as can be seen by rewriting each of the two systems as a fourth-order PDE, yielding bi-Poisson problems with homogeneous boundary conditions. 

  To see identity~\eqref{eq_ttt_representation} consider the test function $\delta\vv = (\delta v,\delta\bQ)\in\VV^\perp$. Employing the definition of the trace operators and~\eqref{eq_dual1} we get that
  \begin{align*}
    \vdual{\Grad\grad v}{\Grad\grad\delta v}_{\cT} + \vdual{v}{\delta v} 
    &= \vdual{\dDiv\Grad\grad v}{\delta v} - \dual{\traceDD{}\Grad\grad v}{\delta v}_\cS + \vdual{v}{\delta v}
    \\
    &= \vdual{-\dDiv \bQ}{\delta v} + \dual{\traceDD{}\bQ}{\delta v}_\cS + \vdual{v}{\delta v}
    \\
    &= \vdual{g}{\delta v} + \dual{\traceDD{}\bQ}{\delta v}_\cS + \vdual{v}{\delta v} \\
    &= b^\perp( (0,0,0,\traceDD{}\bQ), \delta\vv) + \vdual{v+g}{\delta v}
  \end{align*}
  as well as
  \begin{align*}
    \vdual{\dDiv\bQ}{\dDiv\delta\bQ}_{\cT} + \vdual{\bQ}{\delta\bQ} &= \vdual{-g}{\dDiv\delta\bQ}_{\cT} + \vdual{\bQ}{\delta\bQ}
    \\
    &= b^\perp( (g,0,0,0),\delta\vv) + \vdual{\bQ}{\delta\bQ}. 
  \end{align*}
  Therefore,
  \begin{align*}
    \ip{\vv}{\delta\vv}_{\VV^\perp} = b^\perp((g,0,0,\traceDD{}\bQ),\delta\vv) + \vdual{v+g}{\delta v} + \vdual{\bQ}{\delta\bQ}
    \quad\forall \delta\vv \in \VV^\perp. 
  \end{align*}
  To finish the proof it remains to show that $b^\perp(\ww^*,\delta\vv) = \vdual{v+g}{\delta v} + \vdual{\bQ}{\delta\bQ}$. To do so we use the trace operators and~\eqref{eq_dual2} to arrive at
  \begin{align*}
    b^\perp(\ww^*,\delta\vv) &= \vdual{-w^*}{\dDiv\delta\bQ} + \vdual{\bR^*}{\delta\bQ+\Grad\grad \delta v}
    + \dual{\traceGG{} w^*}{\delta\bQ}_{\cS} + \dual{\traceDD{}\bR^*}{\delta v}_{\cS}
    \\
    &= \vdual{\bR^*-\Grad\grad w^*}{\delta\bQ} + \vdual{\dDiv\bR^*}{\delta v} = \vdual{\bQ}{\delta\bQ} + \vdual{v+g}{\delta v}.
  \end{align*}
  This concludes the proof. 
\end{proof}

For the next result we use the auxiliary solution $\uu_\mathrm{a}^\perp = (u_\mathrm{a},\bP_\mathrm{a}^\perp,\widehat u_\mathrm{a}^\perp,\widehat\bp_\mathrm{a}^\perp) \in \UU^\perp$ given by
\begin{align}\label{uwf_Stokes_regularized}
  b^\perp(\uu_\mathrm{a}^\perp,\delta\vv) = L_h(\delta\vv) \quad\forall \delta\vv\in \VV^\perp.
\end{align}
Thus, $u_\mathrm{a}\in H_0^2(\Omega)$ is the solution of~\eqref{Stokes} with regularized data, i.e, 
\begin{align}\label{Stokes_regularized}
  -\rot\Div\Grad\curl u_\mathrm{a} = P_h'\rot\ff \text{ in }\Omega, \quad
  u_\mathrm{a} = 0, \, \curl u_\mathrm{a} = 0 \text{ on }\Gamma.
\end{align}

\begin{prop}\label{prop_dual2}
  Under the assumptions and notations of Proposition~\ref{prop_dual} the solution component $u_\mathrm{a}$ from above and the component $u_h$ of the DPG approximation $\uu_h^\perp$ (see~\eqref{DPG_Stokes}) satisfy
  \begin{align*}
    |\vdual{u_\mathrm{a}-u_h}g| \lesssim \|\uu_\mathrm{a}^\perp-\uu_h^\perp\|_{\UU^\perp}\big( \|\ww-\ww_h\|_{\UU^\perp} + \|\vv-\vv_h\|_{\VV^\perp} \big)
    \quad\forall \ww_h\in \UU_h^\perp, \vv_h\in \VV_h^\perp
  \end{align*}
  with a hidden constant independent of $\cT$.
\end{prop}
\begin{proof}
    For the proof we use the mixed formulation of the DPG method, see,
    e.g.,~\cite[Section 3.2]{Fuehrer_19_SDM} and references therein.
   To that end we define the bilinear form $a^\perp(\cdot,\cdot)$ by
  \begin{align*}
    a^\perp( (\uu,\vv),(\delta\uu,\delta\vv)) = b^\perp(\uu,\delta\vv) + \ip{\vv}{\delta\vv}_{\VV^\perp} - b^\perp(\delta\uu,\vv).
  \end{align*}
  The DPG method~\eqref{DPG_Stokes} then reads
  \begin{align*}
    \uu_h^\perp \in \UU_h^\perp: \qquad 
    a^\perp( (\uu_h^\perp,\mathfrak{e}_h^\perp),(\delta\uu_h,\delta\vv_h)) = L_h(\delta\vv_h) 
    \quad\forall (\delta\uu_h,\delta\vv_h)\in \UU_h^\perp\times \VV_h^\perp.
  \end{align*}
  The element $\mathfrak{e}_h^\perp$ is the error representation function. 
  It satisfies for all $\delta\vv_h\in \VV_h^\perp$
  \begin{align*}
    \ip{\mathfrak{e}_h^\perp}{\delta\vv_h}_{\VV^\perp} = L_h(\delta\vv_h) - b^\perp( \uu_h^\perp,\delta\vv_h) = 
    b^\perp(\uu_\mathrm{a}^\perp-\uu_h^\perp,\delta\vv_h) \lesssim \|\uu_\mathrm{a}^\perp-\uu_h^\perp\|_{\UU^\perp}
    \|\delta\vv_h\|_{\VV^\perp}.
  \end{align*}
  Therefore, choosing $\delta\vv_h = \mathfrak{e}_h^\perp$ we see that 
  \begin{align*}
    \|\mathfrak{e}_h^\perp\|_{\VV^\perp} \lesssim \|\uu_\mathrm{a}^\perp-\uu_h^\perp\|_{\UU^\perp}.
  \end{align*}

  By the ultraweak formulation of the regularized problem~\eqref{uwf_Stokes_regularized} we see that
  \begin{align*}
    a^\perp( (\uu_\mathrm{a}^\perp,\mathfrak{e}_\mathrm{a}^\perp),(\delta\uu,\delta\vv)) = L_h(\delta\vv) 
    \quad\forall (\delta\uu,\delta\vv)\in \UU^\perp\times \VV^\perp
  \end{align*}
  with $\mathfrak{e}_\mathrm{a}^\perp:=0$. Particularly, Galerkin orthogonality holds, i.e., 
  \begin{align*}
    a^\perp( (\uu_\mathrm{a}^\perp-\uu_h^\perp,\mathfrak{e}_\mathrm{a}^\perp-\mathfrak{e}_h^\perp),(\delta\uu_h,\delta\vv_h)) = 0 \quad\forall (\delta\uu_h,\delta\vv_h)\in \UU_h^\perp\times \VV_h^\perp.
  \end{align*}
  Recall the definitions of $\vv$ and $\ww$ from Proposition~\ref{prop_dual}. Let $\ww_h\in \UU_h^\perp$, $\vv_h\in \VV_h^\perp$ be arbitrary. 
  Using~\eqref{eq_ttt_representation} and Galerkin orthogonality we have the representation
  \begin{align*}
    |\vdual{u_\mathrm{a}-u_h}g| &= |b^\perp(\uu_\mathrm{a}^\perp-\uu_h^\perp,\vv)| = |a^\perp( (\uu_\mathrm{a}^\perp-\uu_h^\perp,\mathfrak{e}_\mathrm{a}^\perp-\mathfrak{e}_h^\perp),(\ww,\vv))|
    \\ &=|a^\perp( (\uu_\mathrm{a}^\perp-\uu_h^\perp,\mathfrak{e}_\mathrm{a}^\perp-\mathfrak{e}_h^\perp),(\ww-\ww_h,\vv-\vv_h))|
    \\
    &\lesssim \big(\|\uu_\mathrm{a}^\perp-\uu_h^\perp\|_{\UU^\perp} + \|\mathfrak{e}_\mathrm{a}^\perp-\mathfrak{e}_h^\perp\|_{\VV^\perp}\big) \big(\|\ww-\ww_h\|_{\UU^\perp} + \|\vv-\vv_h\|_{\VV^\perp}\big)
    \\
    &\lesssim \|\uu_\mathrm{a}^\perp-\uu_h^\perp\|_{\UU^\perp} \big(\|\ww-\ww_h\|_{\UU^\perp} + \|\vv-\vv_h\|_{\VV^\perp}\big).
  \end{align*}
  In the last estimate we have used that $\mathfrak{e}_\mathrm{a}^\perp=0$ and $\|\mathfrak{e}_h^\perp\|_{\VV^\perp}\lesssim \|\uu_\mathrm{a}^\perp-\uu_h^\perp\|_{\UU^\perp}$ as elaborated above.
\end{proof}

{\bf Proof of Theorem~\ref{thm_DPG_Stokes}(ii).}
Let $\Pi_h^q\colon L_2(\Omega)\to P^1(\cT)$ denote the orthogonal projection onto piecewise polynomials of degree $\leq q$. 
Writing $u-u_h = u-u_\mathrm{a} + u_\mathrm{a} - \Pi_h^1 u_\mathrm{a} + \Pi_h^1 u_\mathrm{a}-u_h$ the triangle inequality yields
\begin{align*}
  \|u-u_h\| &\leq \|u-u_\mathrm{a}\| + \|u_\mathrm{a} - \Pi_h^1 u_\mathrm{a}\| + \| \Pi_h^1 u_\mathrm{a}-u_h\|.
\end{align*}
For the estimation of the first term we simply use continuous dependence on the data and the properties of the operator $P_h'$ (see the proof of Theorem~\ref{thm_DPG_Stokes}) to conclude that
\begin{align*}
  \|u-u_\mathrm{a}\| \lesssim \|(1-P_h')\rot\ff\|_{-2} \lesssim h^2 \|\rot\ff\|. 
\end{align*}
For the second term we additionally use the approximation properties of the projection and obtain
\begin{align*}
  \|(1-\Pi_h^1)u_\mathrm{a}\| \lesssim h^2 \|u_\mathrm{a}\|_2 \lesssim h^2 \|P_h'\rot\ff\|_{-1} \lesssim h^2 \|\ff\|.
\end{align*}
For the third and final term $\| \Pi_h^1 u_\mathrm{a}-u_h\|$ we employ Proposition~\ref{prop_dual} and~\ref{prop_dual2} with $g=\Pi_h^1 u_\mathrm{a}-u_h$ to see that
\begin{align}\label{eq_g}
  \begin{split}
  \| \Pi_h^1 u_\mathrm{a}-u_h\|^2 &= \vdual{u_\mathrm{a}-u_h}{\Pi_h^1 u_\mathrm{a}-u_h} 
  \lesssim \|\uu_\mathrm{a}^\perp-\uu_h^\perp\|_{\UU^\perp} \big(\|\ww-\ww_h\|_{\UU^\perp} + \|\vv-\vv_h\|_{\VV^\perp}\big)
  \\
  &\lesssim h^s \|\ff\| \big(\|\ww-\ww_h\|_{\UU^\perp} + \|\vv-\vv_h\|_{\VV^\perp}\big) \quad\forall \ww_h\in\UU_h^\perp,\vv_h\in\VV_h^\perp,
  \end{split}
\end{align}
where the last estimate follows as in the proof of Theorem~\ref{thm_DPG_Stokes}. 
It remains to bound 
\begin{align*}
  \|\ww-\ww_h\|_{\UU^\perp} + \|\vv-\vv_h\|_{\VV^\perp}.
\end{align*}
Note that~\eqref{eq_dual1} implies that
\(
  \Delta^2 v = g,
\)
and by the regularity assumption~\eqref{reg_biharmonic} we have $\|v\|_3\lesssim \|g\|_{-1}\lesssim \|g\|$.
Let $\Pi_F^{\dDiv}\colon \HdDiv\cT\to \mathbb{P}^{4,s}(\cT)$ denote the Fortin operator defined in~\cite[Lemma~16]{FuehrerH_19_FDD}. It has the following properties:
\begin{align*}
  \|\Pi_F^{\dDiv}\bR\| & \lesssim \|\bR\| + h^2\|\dDiv\bR\|_{\cT}, 
  \quad
  \dDiv\Pi_F^{\dDiv}\bR = \Pi_h^2\dDiv\bR
  \quad\forall \bR\in\HdDiv\cT.
\end{align*}
Taking $\vv_h = (\Pi_h^2 v,\Pi_h^0\bQ+\Pi_F^{\dDiv}(1-\Pi_h^0)\bQ)\in \VV_h^\perp$ and using the approximation properties of $\Pi_h^q$ as well as the above mentioned properties of the Fortin operator we see that
\begin{align*}
  \|\vv-\vv_h\|_{\VV^\perp} &\lesssim h\|v\|_{3} + \|(1-\Pi_F^{\dDiv})(1-\Pi_h^0)\bQ\| + \|\dDiv(1-\Pi_F^{\dDiv})(1-\Pi_h^0)\bQ\|_\cT \\
  &\lesssim h\|v\|_3 + \|(1-\Pi_h^0)\bQ\| +h^2\|\dDiv(1-\Pi_h^0)\bQ\|_\cT + \|(1-\Pi_h^2)\dDiv(1-\Pi_h^0)\bQ\|_\cT 
  \\
  &\lesssim h\|v\|_3 + h\|\bQ\|_1 + h^2\|g\| + \|(1-\Pi_h^2)g\| = h\|v\|_3 + h^2\|g\| + h\|\Grad\grad v\|_1 \lesssim h\|g\|.
\end{align*}
Here, we have used that $g= \Pi_h^1 u_\mathrm{a}-u_h\in P^1(\cT)$ so that $(1-\Pi_h^2)g=0$. 
For the remaining term in~\eqref{eq_g} we recall from Proposition~\ref{prop_dual2} that $\ww = (g,0,0,\traceDD{}\bQ)+\ww^*$ and bound
\begin{align*}
  \min_{\ww_h\in\UU_h^\perp} \|\ww-\ww_h\|_{\UU^\perp} &\leq \min_{\widehat\bq_h\in \widehat Q_\cS} \|(g,0,0,\traceDD{}\bQ)-(g,0,0,\widehat\bq_h)\|_{\UU^\perp}
  + \min_{\ww_h^*\in\UU_h^\perp} \|\ww^*-\ww_h^*\|_{\UU^\perp}
\end{align*}
which is possible since $(g,0,0,0)\in \UU_h^\perp$. An application of Proposition~\ref{prop_tP_approx} leads to
\begin{align*}
  \min_{\widehat\bq_h\in \widehat Q_\cS} \|(g,0,0,\traceDD{}\bQ)-(g,0,0,\widehat\bq_h)\|_{\UU^\perp} 
  \lesssim h \|\bQ\|_1 + h^{1+s} \|\dDiv\bQ\| \lesssim h \| g\|.
\end{align*}
By the regularity assumption~\eqref{reg_biharmonic}, using system~\eqref{eq_dual2},
and arguing as in Proposition~\ref{prop_DPG_Stokes_L2}, we obtain
\begin{align*}
\min_{\ww_h^*\in\UU_h^\perp} \|\ww^*-\ww_h^*\|_{\UU^\perp} \lesssim h(\|w^*\|_3+\|\bR^*\|_1) + h^{1+s}\|\dDiv\bR^*\|
\lesssim h \|g\|.
\end{align*}
Combining the latter estimates we therefore can bound~\eqref{eq_g} (recalling that $g=\Pi_h^1u_\mathrm{a}-u_h$) further by
\begin{align*}
  \| \Pi_h^1 u_\mathrm{a}-u_h\|^2 \lesssim h^s\|\ff\|\,h\|\Pi_h^1u_\mathrm{a}-u_h\|.
\end{align*}
Overall, we finish the proof of Theorem~\ref{thm_DPG_Stokes}(ii) by concluding that 
\begin{align*}
  \|u-u_h\| \lesssim h^{1+s}(\|\ff\|+\|\rot\ff\|).
\end{align*}

\section{Numerical examples} \label{sec_num}
This section presents some numerical examples for the DPG method for the two-dimensional Stokes problem defined in Section~\ref{sec_Stokes}.
For the built-in error estimator of the DPG method we use the notation
\begin{align*}
  \eta = \sup_{0\neq\vv_h\in\VV_h^\perp} \frac{L_h(\vv_h)-b^\perp(\uu_h^\perp,\vv_h)}{\|\vv_h\|_{\VV^\perp}}.
\end{align*}
We consider three simple benchmark problems.

\subsection{Smooth solution}\label{sec_smoothSolution}
We consider the exact solution $u(x,y) = \sin^2(\pi x)\sin^2(\pi y)$, $(x,y)\in \Omega:=(0,1)^2$
and define the force
\(
  \ff = -\Delta \curl u.
\)
Then, the pair $\bu=\curl u\in\bH_0^1(\Omega)$, $p=0$ satisfies the Stokes equations.
Due to the convexity of $\Omega$, regularity property~\eqref{reg_biharmonic} holds true. 
Theorem~\ref{thm_DPG_Stokes} and Corollary~\ref{cor_DPG_Stokes} predict
\begin{align*}
  \|u-u_h\| = \OO(h^2), \quad \|\bu-\bu_h\|=\OO(h), \quad \|\bP^\perp-\bP_h^\perp\| = \OO(h)
\end{align*}
where $h\simeq \dim(\UU_h^\perp)^{-1/2}$.
Figure~\ref{fig_smoothSolution} supports these theoretical results. 

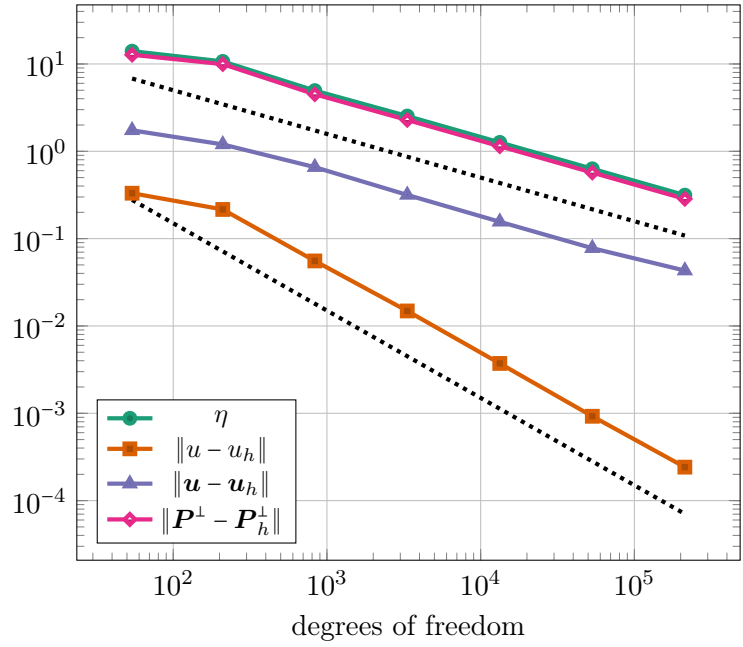
\begin{figure}
  \begin{center}
    \begin{tikzpicture}
      \begin{loglogaxis}[
          width=0.65\textwidth,
          cycle list/Dark2-6,
          cycle multiindex* list={
            mark list*\nextlist
          Dark2-6\nextlist},
          every axis plot/.append style={ultra thick},
          xlabel={degrees of freedom},
          grid=major,
          legend entries={\small $\eta$,\small $\|u-u_h\|$,\small $\|\bu-\bu_h\|$,\small $\|\bP^\perp-\bP_h^\perp\|$},
          legend pos=south west,
        ]
        \addplot table [x=dofDPG,y=estDPG] {data/SmoothSolution.dat};
        \addplot table [x=dofDPG,y=errW] {data/SmoothSolution.dat};
        \addplot table [x=dofDPG,y=errU] {data/SmoothSolution.dat};
        \addplot table [x=dofDPG,y=errM] {data/SmoothSolution.dat};
        \addplot [black,dotted,mark=none] table [x=dofDPG,y expr={50*sqrt(\thisrowno{1})^(-1)}] {data/SmoothSolution.dat};
        \addplot [black,dotted,mark=none] table [x=dofDPG,y expr={15*sqrt(\thisrowno{1})^(-2)}] {data/SmoothSolution.dat};
      \end{loglogaxis}
    \end{tikzpicture}
  \end{center}
  \caption{Error indicator and errors of the field variables for the smooth solution from Section~\ref{sec_smoothSolution}. 
  The black doted lines indicate $\OO(N^{-1/2})=\OO(h)$ resp. $\OO(N^{-1})=\OO(h^2)$ where $N = \dim(\UU_h^\perp)$.}
  \label{fig_smoothSolution}
\end{figure}

\subsection{Lid-driven cavity flow}\label{sec_liddriven}
In this section we consider the classic benchmark problem of a lid-driven cavity flow. 
The cavity is given by $\Omega = (0,1)^2$, with zero external force $\ff$. 
Here, we use  the regularized boundary conditions from~\cite[Section~D.1, Eq.(D.11)]{John_16_FEM},
\begin{align*}
  u|_\Gamma = 0, \quad \partial_{\bn}u(x,y) = 
  \begin{cases}
    \phi(x) & y=1, \\
    0 & \text{else},
  \end{cases}
\end{align*}
where 
\begin{align*}
  \phi(x) = \begin{cases}
    1-\tfrac14\left(1-\cos(\tfrac{0.1-x}{0.1}\pi)\right)^2 & x\in[0,0.1], \\
    1 & x\in(0.1,0.9), \\
    1-\tfrac14\left(1-\cos(\tfrac{x-0.9}{0.1}\pi)\right)^2& x\in[0.9,1].
  \end{cases}
\end{align*}
One verifies that these boundary conditions for $u$ correspond to the boundary conditions
\begin{align*}
  \bu(x,y)|_\Gamma = \begin{cases}
    (\phi(x),0)^\top & y=1, \\
    0 & \text{else}
  \end{cases}
\end{align*}
for the velocity field of the Stokes problem.

\begin{figure}
  \begin{center}
    \begin{tikzpicture}
      \begin{loglogaxis}[
          width=0.49\textwidth,
          cycle list/Dark2-6,
          cycle multiindex* list={
            mark list*\nextlist
          Dark2-6\nextlist},
          every axis plot/.append style={ultra thick},
          xlabel={degrees of freedom},
          grid=major,
          legend entries={\small $\eta$},
          legend pos=south west,
        ]
        \addplot table [x=dofDPG,y=estDPG] {data/LidDriven.dat};
        \addplot [black,dotted,mark=none] table [x=dofDPG,y expr={50*sqrt(\thisrowno{1})^(-1)}] {data/LidDriven.dat};
      \end{loglogaxis}
    \end{tikzpicture}
    \includegraphics[width=0.49\textwidth]{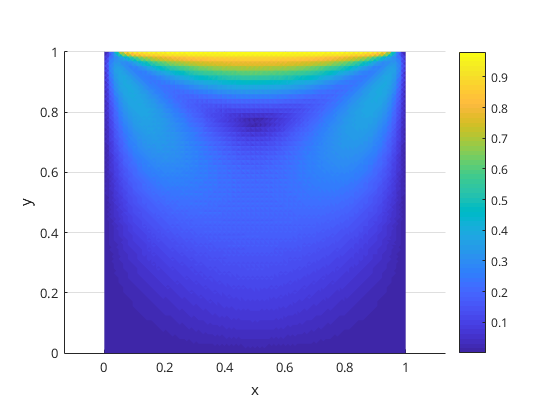}
    \includegraphics[width=0.49\textwidth]{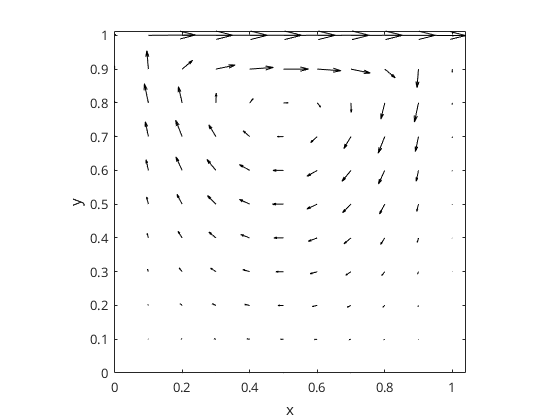}
    \includegraphics[width=0.49\textwidth]{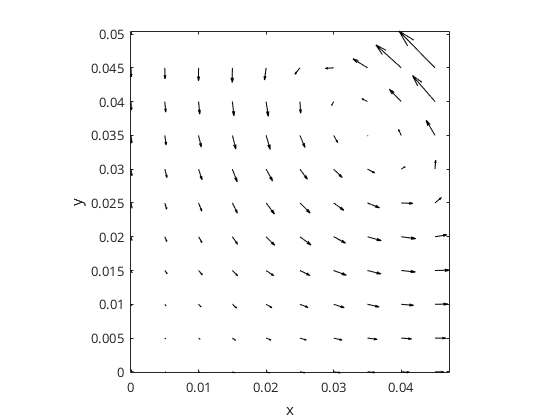}
  \end{center}
  \caption{Upper row: Error indicator $\eta$ for the lid-driven cavity problem (left) and velocity profile $|\bu_h|$ on finest mesh with $\#\cT = 16384$ (right). The black dotted line indicates $\OO(h)$. 
  Lower row: Visualization of the velocity field on the whole domain $\Omega = (0,1)^2$ (left) and magnification on the subdomain $(0,0.05)^2$ (right).}
  \label{fig_liddriven}
\end{figure}

We do not have an explicit representation of the exact solution.
We plot the error estimator $\eta$ in Figure~\ref{fig_liddriven} which indicates that the error converges at the optimal rate after some initial refinements. 
Figure~\ref{fig_liddriven} also shows the magnitude $|\bu_h|$ of the discrete velocity field. 

For the lid-driven cavity problem it is well known that three vortices develop, a big central one and two smaller ones close to the bottom left and right corners (with opposite direction of rotation). 
In the lower row of Figure~\ref{fig_liddriven} we visualize the discrete velocity field by arrows at some equidistributed sample points. The left plot shows the whole domain $\Omega$, whereas the right plot shows (the magnified) discrete velocity field in the subdomain $(0,0.05)^2$, at the the lower left corner.
One observes a vortex with opposite rotation direction compared to the main vortex in the left plot. 

\subsection{Flow in channel with backward step}\label{sec_channel}
In our final example we consider a flow in a channel with a backward facing step,
with domain $\Omega=(0,10)\times (-1,1)\setminus [0,2]\times[-1,0]$ and a parabolic inflow profile.
The setup is similar to~\cite[Section 3.3]{RobertsDM_15_DPG} and references therein.
We consider the following boundary conditions,
\begin{align*}
  \partial_{\bn} u|_\Gamma = 0, \quad
  u(x,y)|_\Gamma = \begin{cases}
    -\frac{(2y+1)(y-1)^2}6 & x=0, \\
    -\frac{(y-1)^2(y+2)}{24} & x=10, \\ 
    0 & y = 1, \\
    -\frac16 & \text{else}.
  \end{cases}
\end{align*}
We stress the fact that the latter condition translates into the following boundary condition for the velocity field,
\begin{align*}
  \bu(x,y)|_\Gamma = \begin{pmatrix}
  u_1(x,y) \\ 0\end{pmatrix}, 
  \quad u_1(x,y) = \begin{cases}
    y(1-y) & x=0, \\
    \frac{(y+1)(1-y)}8 & x=10, \\
    0 &\text{else}.
  \end{cases}
\end{align*}

\begin{figure}
  \begin{center}
    \includegraphics[width=\textwidth]{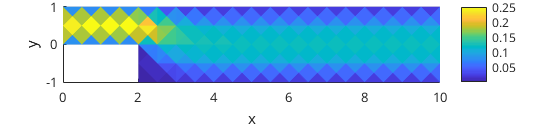}
    \includegraphics[width=\textwidth]{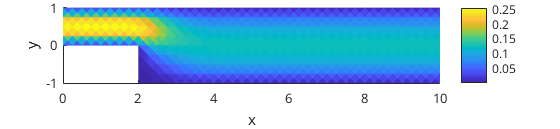}
    \includegraphics[width=\textwidth]{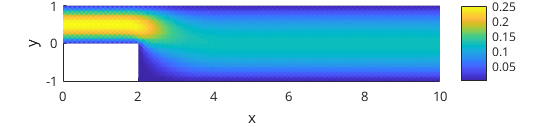}
    \includegraphics[width=\textwidth]{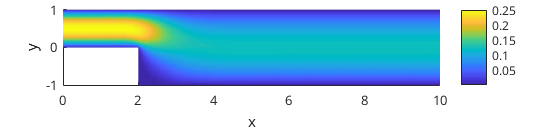}
  \end{center}
  \caption{Velocity magnitude $|\bu_h|$ for the problem of Section~\ref{sec_channel} on a sequence of meshes $\cT_1,\cdots,\cT_4$ with $\#\cT_j = 72\cdot 4^j$ elements.}
  \label{fig_channel}
\end{figure}

Again, we do not know an explicit representation of the solution. 
In Figure~\ref{fig_channel} we visualize $|\bu_h|$ on a sequence of meshes. 
The plots indicate that mass is conserved (not only on fine meshes),
reflecting the fact that mass conservation is intrinsic to formulation~\eqref{Stokes}.


\bibliographystyle{siam}
\bibliography{/home/norbert/tex/bib/bib,/home/norbert/tex/bib/heuer}
\end{document}